\documentclass[a4paper,10pt]{article}
\usepackage[latin1]{inputenc}
\usepackage[T1]{fontenc}                  %only using latex
\usepackage[english]{babel}
\usepackage{verbatim}
\usepackage{amsmath}
\usepackage{amsthm}
\usepackage{amssymb}
\usepackage{graphicx}
\usepackage{bm}
\usepackage{fancyhdr}                  % per lo stile fancy

\usepackage{amsfonts}
\usepackage{color}
\usepackage{graphicx}

\theoremstyle{definition}
\newtheorem{defin}{Definition}[section]

\newtheorem{rem}[defin]{Remark}

\theoremstyle{plane}
\newtheorem{thm}[defin]{Theorem}
\newtheorem{prop}[defin]{Proposition}
\newtheorem{coroll}[defin]{Corollary}
\newtheorem{lemma}[defin]{Lemma}

\newcommand{\vrho}{\varrho}

\newcommand{\wtilde}{\widetilde}
\newcommand{\vphi}{\varphi}

\newcommand{\hra}{\hookrightarrow}

\newcommand{\g}{\gamma}

\newcommand{\eps}{\varepsilon}
\newcommand{\R}{\mathbb{R}}

\newcommand{\N}{\mathbb{N}}
\newcommand{\Z}{\mathbb{Z}}

\renewcommand{\div}{{\rm div}\,}

\newcommand{\Id}{{\rm Id}\,}
\newcommand{\Supp}{{\rm Supp}\,}

\newcommand{\Int}{\displaystyle \int}

\def\ov{\overline}
\def\d{\partial}

\def\dq{\Delta_q}
\def\ddj{\dot \Delta_j}
\def\dj{\Delta_j}
\def\tilde{\widetilde}
\def\hat{\widehat}
\def\div{{\rm div}\,}
\def\s{\sigma}

\def\cC{{\mathcal C}}
\def\cF{{\mathcal F}}

\def\cK{{\mathcal K}}

\def\cP{{\mathcal P}}
\def\cQ{{\mathcal Q}}

\def\cS{{\mathcal S}}

\renewcommand{\div}{{\rm div}\,}
\newcommand{\loc}{{\rm loc}\,}

\def\d{\partial}
\def\div{{\rm div}\,}

\usepackage{enumerate}
\usepackage{xcolor}

\def\ov{\overline}
\def\d{\partial}

\def\dq{\Delta_q}
\def\ddj{\dot \Delta_j}
\def\dj{\Delta_j}

\def\tilde{\widetilde}
\def\hat{\widehat}
\def\div{{\rm div}\,}
\def\s{\sigma}

\def\cC{{\mathcal C}}
\def\cF{{\mathcal F}}

\def\cP{{\mathcal P}}
\def\cQ{{\mathcal Q}}
\def\cS{{\mathcal S}}
\def\cV{{\mathcal V}}

\allowdisplaybreaks

\title{Analysis of an inviscid zero-Mach number system in endpoint Besov spaces for finite-energy initial data }

\author{ Francesco FANELLI and Xian LIAO}
\date\today

\begin{document}

\maketitle

\subsubsection*{Abstract}
{\small
The present paper is the continuation of work \cite{F-L_p}, devoted to the study of
an inviscid  zero-Mach number system in the framework of \emph{endpoint} Besov spaces of type $B^s_{\infty,r}(\R^d)$, $r\in [1,\infty]$, $d\geq 2$, which can be embedded in the Lipschitz class $C^{0,1}$.
In particular, the largest case   $B^1_{\infty,1}$ and the case of H\"older spaces   $C^{1,\alpha}$ are permitted.

The local in time well-posedness result is proved, under an additional $L^2$ hypothesis on the initial inhomogeneity  and velocity field.
A new a priori estimate for  parabolic equations in endpoint spaces $B^s_{\infty,r}$ is presented, which is the key to the proof.

In dimension two, we are able to give a lower bound  for the lifespan, such that the solutions tend to be globally defined when the initial inhomogeneity is small.
There we will show a refined a priori estimate in endpoint Besov spaces   for transport equations with \emph{non solenoidal} transport velocity field.

%However, the global in time existence issue
%for solutions to our system is still an open problem.

%This work is not a trivial adaptation of previous analysis: first of all, it demands new a priori
%estimates for parabolic equations in spaces $B^s_{\infty,r}$. Moreover, for the endpoint space $B^1_{\infty,1}$
%and for the case of space

%otice that  the adopted functional setting includes the case of H\"older spaces of the type
%$C^{1,\alpha}$, and the case of the space $B^1_{\infty,1}$, which is the largest one which is embedded in $C^{0,1}$,
%and so the largest one in which we can expect to recover well-posedness for the system under consideration.
}

\medbreak

\noindent\textbf{Keywords.}
{\small Zero-Mach number system; endpoint Besov spaces; finite energy; well-posedness; parabolic regularity; lifespan.}

\medbreak
\noindent\textbf{Mathematics Subject Classification (2010).}
{\small Primary: 35Q35. Secondary: 76N10, 35B65.}

%%%%%%%%%%%%%%%%%%%%%%%%%%%%%%%%%%%%%%%%%%%%%%%%%%%%%%%%%%%%%%%%%%%%%%%%%%%%%%%%%%%%%%%%%%%%%%%%%%%%%%%%%%%%%%%%%%%%%%%
%%%%%%%%%%%%%%%%%%%%%%%%%%%%%%%%%%%%%%%%%%%%%%%%%%%%%%%%%%%%%%%%%%%%%%%%%%%%%%%%%%%%%%%%%%%%%%%%%%%%%%%%%%%%%%%%%%%%%%%%
\section{Introduction}

In the present paper we will study  the following \emph{inviscid zero-Mach number system}:
\begin{equation}\label{eq:EH}
\left\{ \begin{array}{ccc}
\d_t \rho+\div(\rho v) & =&0,\\
 \d_{t}(\rho v)+\div(\rho v\otimes v ) + \nabla\Pi & = & 0, \\
 \div ( v  + \kappa\rho^{-1}\nabla\rho)& = &0,
\end{array} \right.
\end{equation}
where $\rho=\rho(t,x)\in\R^+$
stands for the mass density, $v=v(t,x)\in\R^d$ for  the velocity field and  $\Pi=\Pi(t,x)$ for the unknown pressure.
The positive heat-conducting coefficient $\kappa=\kappa(\rho)$ depends smoothly on its variable.
The time variable $t$ and the space variable $x$
belong  to $\R^+$ (or to $[0,T]$)  and  $\R^d$, $d\geq2$, respectively.

This model derives from the full compressible, heat-conducting and inviscid system as the Mach number tends to vanish
(see e.g. \cite{Alazard06, Feireisl-N, PLions96,Majda84,Zeytounian04}).
In particular, this singular low-Mach number limit is rigorously justified in Alazard \cite{Alazard06} for smooth enough solutions.
We refer to the introduction of \cite{F-L_p} or the previous literature for more details on the derivation of the system.

Interestingly, System \eqref{eq:EH} can also describe, for instance, a two-component incompressible inviscid mixture with diffusion effects between these two components.
We refer to e.g. \cite{F-K} for more physical backgrounds.

\medbreak
 %%%%%%%%%%%%%%%%%%%%%%%%% HISTORY %%%%%%%%%%%%%%%%%%%%%%%%%%%%%%%%%%%%%%%%%%%%%%%%
%%%%%%%%%%%%%%%%%%%%%%%%%%%%%%%%%%%%%%%%%%%%%%%%%%%%%%%%%%%%%%%%%%%%%%%%%%%%%%%%%%%%%%%%%%%%%%%%%%%%%

Notice  that if we take simply $\kappa\equiv 0$ (i.e. we have no heat conduction), then System \eqref{eq:EH} reduces to the density-dependent
Euler equations
\begin{equation}\label{system:dens-depend}
\left\{
\begin{array}{cc}
&\d_t\rho+\div(\rho v) =0,\\
&\d_t(\rho v)+ \div(\rho v\otimes v)+ \nabla\Pi=0,\\
&\div v=0.
\end{array}
\right.
\end{equation}
We refer to \cite{Antontsev-K-M,B-V1,B-V2}, among other works, for some well-posedness results for System \eqref{system:dens-depend}.

Let us just mention here that, in \cite{D}, Danchin adopted mainly the functional framework of Besov spaces $B^s_{p,r}$, $1<p<+\infty$, which can be embedded in the set of globally Lipschitz functions.
There he  considered e.g.  the finite-energy   initial velocity field case, the case with $p\in [2,4]$ or the case of small inhomogeneity.
All the assumptions  are, roughly speaking, due to the control of (the low frequencies  of) the pressure term.
In \cite{D-F}, Danchin and the first author treated the endpoint   case $B^s_{\infty,r}$.
They also proved a lower bound for the solutions in the case
of space dimension $d=2$: the \emph{infinite} energy data were
considered as well and one has to resort to the analysis the vorticity of the fluid.
%Even if
%the global in time existence issue is still open for system \eqref{system:dens-depend} even for $2$-dimensional flows,
%in \cite{D-F} a bound from below was given for the lifespan of its solutions: in particular, for small perturbation of a constant initial
%density state, the corresponding solution tends to be globally defined.

\medbreak
When the fluid is supposed to be viscous, instead, System \eqref{eq:EH} becomes the viscous zero-Mach number system
\begin{equation}\label{eq:LML}
\left\{ \begin{array}{ccc}
\d_t \rho+\div(\rho v) & =&0,\\
 \d_{t}(\rho v)+\div(\rho v\otimes v ) - \div \sigma+ \nabla\Pi & = & 0, \\
 \div ( v  + \kappa\rho^{-1}\nabla\rho)& = &0,
\end{array} \right.
\end{equation}
where we defined the viscous stress tensor
$$
\sigma =2\zeta Sv+\eta\div v\Id, \quad Sv:=\frac 12 (\nabla v+(\nabla v)^T),
$$
for two positive viscous coefficients $\zeta$, $\eta$.

The viscous system \eqref{eq:LML} is the low-Mach number limit system of the full Navier-Stokes equations, and hence it describes for instance
the motion of highly subsonic ideal gases.
See \cite{Alazard06, BdV,   Embid87,Kazhikhov-Smagulov,L, Secchi88}  and references therein for further results.
Let us just mention that, Danchin and the second author \cite{D-L} addressed the well-posedness issue in the general \emph{critical} Besov spaces  $B^{d/{p_1}}_{p_1,1}\times B^{d/{p_2}-1}_{p_2,1}$, with technical restrictions on the Lebesgue exponents $p_1, p_2$.
Under a special relationship between the viscous coefficient and the heat-conduction (or diffusion) coefficient, \cite{L} showed the global-in-time wellposedness result in dimension two.

%%%%%%%%%%%%%%%%%%%%%
%%%%%%%%%%%%%%%%%%%%
%\red{EXPLAIN A BIT RESULTS OF \cite{D-L}}
%%%%%%%%%%%%%%%%%%%
%%%%%%%%%%%%%%%%%%%%%%%

\medbreak
To our knowledge, there are just few well-posedness results for the inviscid zero-Mach number system \eqref{eq:EH}.
We refer here that  in \cite{B-S-V},   Beir\~ao da Veiga, Serapioni and Valli proved existence of classical solutions on smooth bounded domains for our system \eqref{eq:EH}.

In our previous work \cite{F-L_p}, instead, we investigate the well-posedness in the functional framework of Besov spaces.
There, we reformulated System \eqref{eq:EH} by introducing  new \emph{divergence-free} velocity field.
Similarly, let us immediately perform an \emph{invertible} change of unknowns here, to  introduce the set of  equations (see \eqref{system} below) we will mainly work on:
for the details we refer  to \cite{D-L, F-L_p}.
 
For notational simplicity, we introduce three coefficients, $a=a(\rho)$,  $b=b(\rho)$ and $\lambda=\lambda(\rho)$, such that
\begin{equation}\label{relation:a,b}
\nabla a\,=\,\kappa\nabla\rho
\,=\,-\rho\nabla b,\quad \lambda=\rho^{-1}>0, \quad a(1)=b(1)=0.
\end{equation}
Then, we introduce the new divergence-free ``velocity'' $u$ and the new ``pressure'' $\pi$   as
\begin{equation}\label{relation:u}
u    \,:=\,v+\kappa\rho^{-1}\nabla\rho\,=\,v-\nabla b,
\qquad \pi\,=\,\Pi-\kappa\d_t\rho \,=\,\Pi-\d_t a\,.
\end{equation}
Therefore, System \eqref{eq:EH} can be rewritten as the following system for the unknowns $(\rho, u, \pi)$:
\begin{equation}\label{system}
\left\{
\begin{array}{cc}
&\d_t\rho+u\cdot\nabla\rho-\div(\kappa\nabla\rho)=0,\\
&\d_t u+(u+\nabla b)\cdot\nabla u+\lambda\nabla\pi=h,\\
&\div u=0,
\end{array}
\right.
\end{equation}
where the new nonlinear ``source'' term $h$ reads as
\begin{equation}\label{h}
h(\rho,u)\,=\,\rho^{-1}\div (v\otimes \nabla a)\,=\,
-u\cdot\nabla^2 b-(u\cdot\nabla\lambda)\nabla a-(\nabla b\cdot\nabla\lambda)\nabla a-\div(\nabla b\otimes\nabla b)\,.
\end{equation}

In the above mentioned work \cite{F-L_p}, we studied the  well-posedness of the zero-Mach number system, in its
reformulated version \eqref{system}, in the setting of Besov spaces $B^s_{p,r}(\R^d)$ embedded in the class $C^{0,1}$
of globally Lispchitz functions, that is to say for
$$
s\,>\,1\,+\,\frac{d}{p}\,,\qquad\qquad\mbox{ or }\qquad\qquad s\,=\,1\,+\,\frac{d}{p}\quad\mbox{ and }\quad r\,=\,1\,. \leqno(C)
$$
Such a restriction is in fact necessary, essentially due to the transport equation for the velocity field: preserving the initial regularity
demands $u$ to be at least locally Lipschitz with respect to the space variable. On the other hand, the non-linear source term
$h$ requires the control of this Besov norm on $\nabla^2\rho$: this is guaranteed by the smoothing effect of
of the \emph{parabolic equation} for the density.
{}
Due to technical reasons,   we had to impose the additional condition in \cite{F-L_p}
\begin{equation} \label{eq:p}
p\,\in\,[2,4]\,.
\end{equation}
This hypothesis \eqref{eq:p} ensures that the ``source'' term $h$, composed of quadratic terms, belongs to $L^2(\R^d)$.
 Hence, regarding the pressure which   satisfies an elliptic equation in divergence form
 $$
 \div(\lambda\nabla\pi)=\div(h-(u+\nabla b)\cdot\nabla u),
 $$
  the pressure term $\nabla\pi$ belongs to $L^2$ too.
 This gives control on the low frequencies of the pressure term.

Finally, under conditions $(C)$ and \eqref{eq:p}, we proved local in time well-posedness of System \eqref{system} in spaces $B^s_{p,r}$,
as well as a continuation criterion for its solutions   and a bound from below for the lifespan in any space dimension $d\geq2$.

\medbreak
In the present paper we propose a different study, rather in endpoint Besov spaces $B^s_{\infty,r}$ which still verifies condition $(C)$ (with $p=+\infty$ of course), in the same spirit of work \cite{D-F}.
This functional framework includes, in particular, the case of H\"older spaces of type $C^{1,\alpha}$, and the case of $B^1_{\infty,1}$,
which is the largest one embedded in the space of globally Lipschitz functions, and so the largest one in which
one can expect to recover well-posedness for our system.

We will add a \emph{finite-energy} hypothesis on the initial data, which is fundamental to control the pressure term, just as the above condition \eqref{eq:p} assumed in \cite{F-L_p}.

Then we are able to prove the local in time well-posedness issue for System \eqref{system} in the adopted functional framework.
The key point of the analysis is the proof of new a priori estimates for parabolic equations in spaces $B^s_{\infty,r}$
(see Proposition \ref{prop:heat-holder}).  
%This will be done firstly in H\"older Spaces where the variable coefficient in second-order term can be viewed as a perturbation locally.
%Then by Littlewood-Paley decomposition the result can be generalized to any $r\in [1,\infty]$.
%The paradifferential calculus will be used intensively.
%This will be done can be done by use of a Microlocal Analysis technique, performing a double
%localization first in the physical space (by a locally finite covering of $\R^d$), and then in the Fourier space
%(by use of Littlewood-Paley decomposition operators).

The global in time existence of solutions to the inviscid zero-Mach number system is still an open problem, even in the simpler case
of space dimension $d=2$. However, similarly as in \cite{D-F},
we are able to move a first step in this direction: by establishing an explicit lower bound for the lifespan of the solutions
in dimension $d=2$, we show that planar flows tend to be globally defined if the inital density is ``close'' (in an appropriate sense)
to a constant state.
Such a lower bound improves the one stated in \cite{F-L_p}, and it can be proved resorting to
arguments similar as in Vishik \cite{V} and Hmidi-Keraani \cite{H-K-2008}.
More precisely, the \textit{scalar vorticity} satisfies a transport equation, and then one aims at bounding it
\textit{linearly} with respect to the velocity field.
Unluckily, in our case the transport velocity occurring in the vorticity equation is the original vector-field $v$ of
System \eqref{eq:EH}, which is \textit{not} divergence-free:
hence, one can just bound the vorticity linearly in $v$ and $\div v$ (see Proposition \ref{p:transp_0}). 
Since the potential part of $v$ just depends on the  density term $\rho$, for which parabolic effect gives enough regularity
to control $\div v$.

\smallbreak
Let us conclude the introduction by pointing out that we decided to adopt the present functional framework, i.e. $B^s_{\infty,r}\cap L^2$,
just for simplicity and clarity of exposition. Actually, combining the techniques of \cite{F-L_p} with the ones in \cite{D},
it's easy to see that our results can be extended to any space $B^s_{p,r}$ which satisfies condition $(C)$ for any $1<p\leq +\infty$.

\bigbreak
Before going on, we give  an overview of the paper.

The next section is devoted to the statement of our main results.

In Section \ref{s:tools} we briefly present the tools we use in our analysis, namely Littlewood-Paley decomposition
and paradifferential calculus, while in Section \ref{s:est_Besov} we prove fundamental a priori estimates for parabolic
and transport equations in endpoint Besov spaces.

Finally, Section \ref{s:proofs} contains the proof of our results.

%%%%%%%%%%%%%%%%%%%% Some acknowledgements %%%%%%%%%%%%%%%%%%%%%%%%% Fanu wants to put some...
\subsubsection*{Acknowledgements}

%%%%%%%%%%%%%%%%%%%%%%% OLD INSTITUTIONS %%%%%%%%%%%%%%%%%%%%%%%%%%%%%%%%%%%%%%%%%%%%%%%%%

The authors are deeply grateful to their previous institutions, the Laboratoire d'Analyse et de Math\'ematiques Appliqu\'ees -- UMR 8050,
Universit\'e Paris-Est and BCAM - Basque Center for Applied Mathematics, which they belonged to when the work started.

%%%%%%%%%%%%%%%%%%%%%%%%%%%%%%%%%%%%%%%%%%%%%%%%%%%%%%%%%%%%%%%%%%%%%%%%%%%%%%%%%%%%%%%%%%%%%%%
The first author was partially supported by Grant MTM2011-29306-C02-00, MICINN, Spain,
ERC Advanced Grant FP7-246775 NUMERIWAVES, ESF Research Networking Programme OPTPDE and Grant PI2010-04 of the Basque Government.
During the last part of the work, he was also supported by the project ``Instabilities in Hydrodynamics'',
funded by the Paris city hall (program ``\'Emergences'') and the
Fondation Sciences Math\'ematiques de Paris.

The second author  was partially supported by the  project ERC-CZ
LL1202, funded by the Ministry of Education, Youth and Sports of the Czech Republic.

%%%%%%%%%%%%%%%%%%%%%%%%%%%%%%%%%%%%%%%%%%%%%%%%%%%%%%%%%%%%%%%%%%%%%%%%%%%%%%%%%%%%%%%%%%%%

The first author is member of the Gruppo Nazionale per l'Analisi Matematica, la Probabilit\`a e le loro Applicazioni (GNAMPA) of the
Istituto Nazionale di Alta Matematica (INdAM).

\section{Main Results} \label{s:results} %%%%%%%%%%%%%%%%%%%%%%%%%
 %%%%%%%%%%% SYSTEM %%%%%%%%%%%%%%%%%%%%%%%%%%%%%%%%%%%

%%%%%%%%%%%%%%%%%%%%%

As explained in the introduction, in the sequel we will deal with system \eqref{system}-\eqref{h} in endpoint Besov spaces
$B^s_{\infty,r}$  with the indices $s\in\R$ and $r\in[1,+\infty]$ satisfying $(C)$ (for $p=+\infty$), i.e.
\begin{equation} \label{eq:cond_index}
 s\,>\,1\qquad\qquad\qquad\mbox{ or }\qquad\qquad\qquad s\,=\,r\,=\,1\,.
\end{equation}
Recall that this  is sufficient to ensure the embedding $B^s_{\infty,r}\,\hra\,C^{0,1}$.

In order to ensure  the velocity field $u$ to belong to $B^s_{\infty,r}(\R^d)$,   the source term $h$  in the velocity equation, which  involves   two derivatives of the
density $\nabla^2\rho$, should be in the same space.
 Nonetheless, new a priori estimates for parabolic equations in endpoint Besov spaces $B^s_{\infty,r}$ (see
Proposition \ref{prop:heat-holder}  below) will guarantee the gain of two orders of regularity for the density  as time goes by.
For this reason we take the initial inhomogeneity $\varrho_0:=\rho_0-1\,\in\,B^s_{\infty,r}$.
{}
And hence we will    get the density in the so-called Chemin-Lerner space $\tilde L^\infty_T(B^s_{\infty,r})$ and $\tilde L^1_T(B^s_{\infty,r})$.
See Definition \ref{def:Besov,tilde} for the definition of these time-dependent Besov spaces.

Moreover, in order to avoid vacuum regions, we will always suppose that the initial density satisfy
$$
0\,<\,\rho_\ast\,\leq\,\rho_0\,\leq\,\rho^\ast\,.
$$ 
By applying  \textit{maximum principle} on the parabolic equation $\eqref{system}_1$, one gets a priori that the density $\rho$
(if it exists on the time interval $[0,T]$) keeps the same  upper and lower bounds as the initial density $\rho_0$:
\begin{equation*}\label{bound}
0<\rho_\ast\leq \rho(t,x) \leq \rho^\ast,\quad \forall\, t\in [0,T],\, x\in \R^d.
\end{equation*}
Hence, applying  the divergence operator to Equation $\eqref{system}_2$
gives an elliptic equation for $\pi$ of the form
\begin{equation}\label{eq:pi}
\div(\lambda \nabla\pi)= \div (h-v\cdot\nabla u)\,,\qquad\mbox{ with }\quad\lambda=\lambda(\rho)\geq \lambda_\ast:=(\rho^\ast)^{-1}>0\,.
\end{equation}
By a result in \cite{D}, we hence have a priori energy estimate for $\nabla\pi$ (independently on $\rho$):
\begin{equation*}\label{eq:el0}
\lambda_\ast \|\nabla\pi\|_{L^2}\leq\|h-v\cdot\nabla u\|_{L^2}\,.
\end{equation*}
This gives \textit{low frequency} informations for $\nabla\pi$.

One then considers  the following \textit{energy estimates}. %, which will ensure  $\nabla\pi\in L^1([0,T]; L^2)$.
First of all, the mass conservation law    $\eqref{system}_1$  entails   (provided $u\in L^\infty_T(L^\infty)$)
\begin{equation}\label{L2:rho}
\frac{1}{2}  \Int_{\R^d} |\rho(t)-1|^2 \, +\, \Int^t_0\Int_{\R^d} \kappa |\nabla\rho|^2
\,=\, \frac{1}{2} \|\rho_0-1\|_{L^2}^2.
\end{equation}
Now we rewrite the momentum conservation law  $\eqref{system}_2$ into
\begin{equation}\label{eq:u}
\rho\d_t u+\rho v\cdot\nabla u +\nabla\pi=\div(v\otimes\nabla a).
\end{equation}
Then, using equation $\eqref{eq:EH}_1$ and $\div u=0$, taking the $L^2$ scalar product of the  previous relation by $u$ entails
\begin{equation}\label{L2:velocity}
 \frac{1}{2}\frac{d}{dt} \Int_{\R^d} \rho |u|^2
 \,\equiv\,
\Int_{\R^d} \left(\rho \d_t u+ \rho v\cdot\nabla u \right)\cdot u
=\langle \div(v\otimes \nabla a), u\rangle_{L^2(\R^d)}.
\end{equation}
Recalling the definitions of $a$ and $b$ in \eqref{relation:a,b}, one bounds the above right-hand side by
(up to a multiplicative constant depending on $\rho_\ast$ and $\rho^\ast$)
\begin{align*}
 \bigl(\Theta'(t)\,\|u\|_{L^2}^2\,+\,\|\nabla\rho\|_{L^2}^2\bigr)\,,\;\mbox{ with }\;
  \Theta(t):=\Int^t_0 \left(\|\nabla\rho\|_{L^\infty}^2\,+\, \|\nabla\rho\|_{L^\infty}^4
                   \,+\,\|\nabla^2\rho\|_{L^\infty}     \,+\, \|\nabla^2\rho\|_{L^\infty}^2\right)d\tau\,.
\end{align*}
Hence  if $(\rho_0-1, u_0) \in L^2$ and $\Theta(T)<+\infty$ (this will be ensured by Besov regularity),
then we gather
$$u,\rho-1\,\in\,L^\infty_T(L^2)\,,\quad
 \nabla\rho\,\in\, L^2_T(L^2)\quad
 \mbox{ and hence }\quad h  \in\, L^1 ([0,T];L^2), \nabla\pi\in L^1 ([0,T];L^2).
 $$

To conclude,   we have the following local-in-time wellposedness result for System \eqref{system}.
\begin{thm}\label{thm:L2wp}
Let $d\geq2$ an integer and take $s\in\R$ and $r\in[1,+\infty]$ satisfying condition \eqref{eq:cond_index}.

Suppose that the initial data  $(\rho_0, u_0)$ fulfill
\begin{equation}\label{L2:initial data}
\rho_0-1\,,\,u_0\;\in\;B^s_{\infty,r}(\R^d)\cap L^2\;,\qquad
\rho_0\,\in\,[\rho_\ast, \rho^\ast]\;,\qquad \div u_0\,=\,0\,.
\end{equation}
Then there exist a positive time $T$  and a unique solution
$(\rho,u,\nabla\pi)$ to System \eqref{system} such that   $(\vrho,u,\nabla\pi):=(\rho-1,u,\nabla\pi)$  belongs
to the space $E^s_r(T)$,  defined as the set of  triplet $ (\varrho, u, \nabla \pi) $ such that
\begin{equation}\label{space:E}
\left\{\begin{array}{c }
\varrho\,\in\,\wtilde C\bigl([0,T];B^{s}_{\infty,r}\bigr)\cap\wtilde L^{1}\bigl([0,T];B^{s+2}_{\infty,r}\bigr)\cap C\bigl([0,T];L^2\bigr)\,,
\quad \rho_\ast\,\leq\,\varrho+1\,\leq\,\rho^\ast\,, \\[1ex]
\nabla\varrho\,\in\,L^2\bigl([0,T];L^2\bigr)\,, \\[1ex]
u \,\in\,\wtilde{C}\bigl([0,T];B^{s}_{\infty,r}\bigr)^d\,\cap\,C\bigl([0,T];L^2\bigr)^d\,, \\[1ex]
\nabla\pi \, \in \, \wtilde{L}^{1}\bigl([0,T];B^{s}_{\infty,r}\bigr)^d\cap L^1\bigl([0,T];L^2\bigr)^d\,,
\end{array}\right.
\end{equation}
with $\tilde C_{w}([0,T];B^s_{p,r}) \hbox{ if }\,r=+\infty $.
\end{thm}

\begin{rem}\label{rem:origin}
Let us remark the well-posedness result for the original system \eqref{eq:EH}.
 According to  the change of variables \eqref{relation:u}, one knows
$$u=\cP v,\quad \nabla b=\cQ v\,,\qquad
\mbox{ where }\quad \hat{\cQ u}(\xi) = -  (\xi/{|\xi|^2})\,\xi\cdot\hat u(\xi),
\quad \cP=I-\cQ.
$$
Hence, for the original system  \eqref{eq:EH}, if the initial datum $(\rho_0, v_0)$ satisfies
$$
0<\rho_\ast\leq \rho_0\leq \rho^\ast\,,\quad \nabla b(\rho_0)=\cQ v_0,
\qquad \rho_0-1, \cP v_0\in B^s_{\infty,r} \cap L^2  ,
$$
then there exist a $T>0$ and a unique solution $(\rho,v,\nabla\Pi)$ to System \eqref{eq:EH} such that
$\rho_\ast\leq \rho\leq\rho^\ast$ and
\begin{equation*}
\left\{\begin{array}{c }
\varrho=\rho-1 \in  \wtilde C\bigl([0,T];B^s_{\infty,r}\bigr)\cap \wtilde L^{1}\bigl([0,T];B^{s+2}_{\infty,r}\bigr)\cap
C\bigl([0,T];L^2\bigr)\,,
\quad \nabla\varrho\in L^2\bigl([0,T];L^2\bigr)\,, \\[1ex]
\cP v   \in \wtilde C\bigl([0,T];B^s_{\infty,r}\bigr) \cap C\bigl([0,T];L^2\bigr)\,,
\quad v\in \wtilde C\bigl([0,T]; B^{s-1}_{\infty,r}\bigr)\,, \\[1ex]
\quad \nabla\Pi   \in   \wtilde{L}^{1}\bigl([0,T];B^s_{\infty,r}\bigr) \cap L^1\bigl([0,T];L^2\bigr)\,,
\end{array}\right.
\end{equation*}
with $\tilde C_{w}([0,T];B^s_{p,r}) \hbox{ if }\,r=+\infty $.

\end{rem}

\begin{rem}\label{rem:any_p}
As said in the  introduction, we can replace the Besov space $B^s_{\infty,r}$ in Theorem \ref{thm:L2wp}
by any general Besov space $B^s_{p,r}$, $p\in\,]1,+\infty]$ such that condition $(C)$ is fulfilled.

The proof is quite standard, and it goes along the lines
of the one in \cite{F-L_p}, with suitable modifications corresponding to the finite energy conditions. One can refer also to paper \cite{D}, where an analogous result is proved
for the density-dependent Euler equations.

%For the sake of conciseness, we omit the proof here for general $1<p<+\infty$, and we focus only in the case $p=+\infty$.
%In the end, this is the most important one, in view of Theorem \ref{th:2D_life} below.
\end{rem}

If   $\rho\equiv 1$,   System \eqref{system} becomes the classical  Euler system.
For this system, the global-in-time existence issue in dimension $d=2$ has been well-known since 1933, due to the pioneering
work Wolibner \cite{W}.
For non-homogeneous perfect fluids, see system \eqref{system:dens-depend}, it's still open if its solutions
exist globally in time. However, in \cite{D-F} it was proved that, for
initial densities close to a constant state, the lifespan of the corresponding solutions tends to \textit{infinity}.
In analogy, we have the following result for our system.
\begin{thm} \label{th:2D_life}
Let $d=2$, and let us assume the hypotheses of Theorem \ref{thm:L2wp}.

Then there exist  $\ell>5$ and $L>0$ (depending only on $\rho_\ast, \rho^\ast, s, r$) such that the
lifespan of the solution to System \eqref{system}, given by Theorem \ref{thm:L2wp}, is bounded from below by the quantity
\begin{equation} \label{est_life_2D:lifespan}
\frac{  L}{ \Gamma_0 }\,
\log\left(\frac{  L}{ \Gamma_0^2 }\log
\left( 1+\frac{  L}{\left(1+\|\vrho_0\|^\ell_{B^1_{\infty,1}}\right)
\|\vrho_0\|_{B^1_{\infty,1}}}\right)\right),
\end{equation}
where we defined $\Gamma_0=1+  \|\vrho_0\|^2_{L^2}+\|u_0\|_{L^2\cap B^1_{\infty,1}}$.
\end{thm}

\begin{rem}
One easily see from \eqref{est_life_2D:lifespan} that if the initial density $\rho_0$ tends to
$1$, then this lower bound tends to infinity, which means the solutions tends to exist all the time.
\end{rem}

\begin{rem}
We can just consider the limit Besov space norm $B^1_{\infty,1}$ in the statement
of Theorem \ref{th:2D_life}, and we will concentrate only on the  $B^1_{\infty,1}$ case in the proof.
In fact, similar as in the proof of the continuation criterion in \cite{F-L_p}, by classical commutator estimates and product estimates (see Proposition \ref{p:comm} and Proposition \ref{c:op}), one knows that if, on the time interval $[0,T^\ast]$, $T^\ast<+\infty$, one has
$$
\|(\nabla\rho, u)\|_{L^\infty_{T^\ast}(L^\infty)}
+
\int^{T^\ast}_0
\Bigl(  1+\|\nabla u\|_{L^\infty}  +\|\nabla\varrho\|_{L^\infty}^4
+\|\nabla^2\varrho\|_{L^\infty}^2 +\|\nabla\pi\|_{L^\infty}\Bigr) <+\infty,
$$
then the solution $(\varrho, u)$ with the initial data $(\varrho_0, u_0)\in B^s_{\infty,r}$ will be well defined in the solution space $E^s_r(T^\ast)$.
On the other side, the above finiteness condition can be ensured if one already has the solution defined in the limit solution space $E^1_{1}(T^\ast)$.
%$$
%(\varrho, u, \nabla\pi)\in 
%\bigl(L^\infty([0,T^\ast]; B^1_{\infty,1})\cap L^1([0,T^\ast]; B^3_{\infty,1}) \bigr)
%\times \bigl(L^\infty([0,T^\ast];B^1_{\infty,1})\bigr)^d
%\times \bigl(L^1([0,T^\ast];B^1_{\infty,1})\bigr)^d.
%$$
%

%Standard continuation criterions (see e.g. \cite{F-L_p} for our system, or \cite{D, D-F} for non-homogeneous inviscid incompressible
%fluids) imply that it's not restrictive to consider just the endpoint regularity $B^1_{\infty,1}$ in the statement
%of Theorem \ref{th:2D_life}.

%%%%%%%%%%%%%%%%%
%%%%%%%%%%%%%%%%%%%%%%%%%%%%%%%%%%%%%%%%%%%%%%%
%\red{WHAT ABOUT CONTINUATION CRITERION ? IT WOULD BE OK TO EXPLAIN WHY THEOREM \ref{th:2D_life} IN THAT SETTING...}
%%%%%%%%%%%%%%%%%%%%
\end{rem}

%%%%%%%%%%%%%%%%%%%%%%%%%%%%%%%%55%%%%%%%%%%%%%%%%%%%%%%%%%%%%%%%%%%%%%%
%%%%%%%%%%%%%%%%%%%%%%%%%%%%%%%%%%%%%%%%%%%%%%%%%%%%%%%%%%%%%%%%%%%%%%%%%

Before going on, let us introduce some notations.
We agree that in the sequel, $C$ always denotes some ``harmless'' constant depending only on $d, s, r, \rho_\ast,\rho^\ast$,
unless otherwise defined. Notation $A\lesssim B$ means $A\leq C B$ and $A\sim B$ says $A$ equals to $B$, up to a constant factor.
For notational convenience, we denote
$$\vrho=\rho-1.$$
%%%%%%%%%%%%%%%%%%%%%%%%%%%%%%%%%%%%%%%%%%%%%%%%%%%%%%%%%%%%%%%%%%%%%%%%
%%%%%%%%%%%%%%%%%%%%%%%%%%%%%%%%%%%%%%%%%%%%%%%%%%%%%%%%%%%%%%%%%%%%%%%5

%%%%%%%%%%%%%%%%%%%%%%%%%%%%%%%%%%%%%%%%%%%%%%%%%%%%%%%%%%%%%%%%%%%%%%%%%%%%%%%%%%%%%%%%%%%%%%%%%%%%%%%%%%%%%%%%
\section{A brief review of Fourier analysis} \label{s:tools}
In this section, we recall  some   definitions and results in  Fourier analysis which will be used in this paper.
Unless otherwise specified, all the presentation in this section have been proved in \cite{B-C-D}, Chapter 2.

Firstly, let's recall the Littlewood-Paley decomposition.
Fix a smooth radial function
$\chi$ supported in   the ball $B(0,\frac43),$ such that it
equals to $1$ in a neighborhood of $B(0,\frac34)$
and    is nonincreasing
over $\R_+$.
Define
$\varphi(\xi)=\chi(\frac\xi2)-\chi(\xi).$
{}
The {\it non-homogeneous dyadic blocks} $(\Delta_j)_{j\in\Z}$
 are defined by\footnote{In what follows we agree  that  $f(D)$ stands for
the pseudo-differential operator $u\mapsto\cF^{-1}(f(\xi) \cF u(\xi)).$}
$$
\dj:=0\ \hbox{ if }\ j\leq-2,\quad\Delta_{-1}:=\chi(D)\quad\hbox{and}\quad
\Delta_j:=\varphi(2^{-j}D)\ \text{ if }\  j\geq0.
$$
We  also introduce the following low frequency cut-off operators:
$$
S_ju:=\chi(2^{-j}D)=\sum_{j'\leq j-1}\Delta_{j'}\quad\text{for}\quad j\geq0,\quad S_j u\equiv 0 \quad\text{for}\quad j\leq0.
$$
%%%%%%%%%%% DEFINITION - BESOV %%%%%%%%%%%%%%%
One hence defines \emph{non-homogeneous Besov space} $B^s_{p,r}$ as follows:
\begin{defin}
\label{def:besov}
  Let  $u\in \cS'$, $s\in \R$, $(p,r)\in [1,\infty]^2.$ We set
$$
\|u\|_{B^s_{p,r}}:=\bigg(\sum_{j} 2^{rjs}
\|\Delta_j  u\|^r_{L^p}\bigg)^{\frac{1}{r}}\ \text{ if }\ r<\infty
\quad\text{and}\quad
\|u\|_{B^s_{p,\infty}}:=\sup_{j}\left( 2^{js}
\|\Delta_j  u\|_{L^p}\right).
$$
The space $B^s_{p,r}$ is  the subset of  tempered distributions $u$ such  that
$\|u\|_{B^s_{p,r}}$ is finite.
\end{defin}

Recall that, for all $s\in\R$, we have the equivalence $H^s\equiv B^s_{2,2}$, while for all $s\in\,\R_+\!\!\setminus\!\N$,
the space $B^s_{\infty,\infty}$ is actually the H\"older space $ {C}^s$.
If $s\in\N$, instead, we set $ {C}^s_*:=B^s_{\infty,\infty}$, to distinguish it from the space $ {C}^s$ of
the differentiable functions with continuous partial derivatives up to the order $s$. Moreover, the strict inclusion
$ {C}^s_b\,\hra\, {C}^s_*$ holds, where $ {C}^s_b$ denotes the subset of $ {C}^s$ functions bounded with all
their derivatives up to the order $s$.
Finally, for $s<0$, the ``negative H\"older space'' $ {C}^s$ is defined as the Besov space $B^s_{\infty,\infty}$.

\medbreak
 For spectrally localized functions, one has the following   \emph{Bernstein's inequalities}:
  \begin{lemma}\label{lpfond}
There exists a $C>0$  such that, for any $k\in \Z^+$, $\lambda\in\R^+$,  $(p,q)\in [1,\infty]^2$ with  $p\leq q$, then
$$
\displaylines{
{\rm Supp}\, \widehat u \subset   B(0,\lambda )
\Longrightarrow
\|  u\|_{L^q} \leq
 C \lambda^{ d(\frac{1}{p}-\frac{1}{q})}\|u\|_{L^p};
 \cr
{\rm Supp}\, \widehat u \subset \{\xi\in\R^N\,/\,  \lambda\leq|\xi|\leq 2\lambda\}
\Longrightarrow C^{-k-1}\lambda^k\|u\|_{L^p}
\leq
\|\nabla^k u\|_{L^p}
\leq
C^{k+1}  \lambda^k\|u\|_{L^p}.
}$$
\end{lemma}
We remark explicitly that by previous lemma one has, for any $f\in L^2$,
$$
\|\Delta_{-1}f\|_{L^\infty}\leq C\|\Delta_{-1}f\|_{L^2}\leq C\|f\|_{L^2}.
$$

  One also has the following embedding and interpolation results:
  \begin{prop}\label{c:embed}
  Space $B^{s_1}_{p_1,r_1}$ is continuously embedded in Space $B^{s_2}_{p_2,r_2}$ whenever
  $1\leq p_1\leq p_2\leq\infty$ and
  $$
  s_2< s_1-d/p_1+d/p_2\quad\hbox{or}\quad
  s_2=s_1-d/p_1+d/p_2\ \hbox{ and }\ 1\leq r_1\leq r_2\leq\infty.
  $$
Moreover, one  has the following interpolation inequality:
$$
  \|\varrho\|_{B^{s+1}_{\infty,r}}\,\leq\,C\|\varrho\|_{B^s_{\infty,r}}^{1/2}\, \|\varrho\|_{B^{s+2}_{\infty,r}}^{1/2}\;,
\qquad  \qquad
  \|\nabla\pi \|_{  B^{s-1}_{\infty,r} }\,\leq\,
C\|\nabla\pi \|_{  L^2 }^{\g}\,\|\nabla\pi \|_{ B^s_{\infty,r} }^{1-\g}\quad(0<\g<1)\,.
$$
  \end{prop}

 One also has the following classical  commutator estimate: 
 \begin{prop}\label{p:comm}
  If $s>0$, $r\in [1,\infty]$, then there exists a constant $C$ depending only on $d,s,r$ such that
\begin{equation}\label{linest:comm}
\int^t_0\left\|2^{j s}\left\|[\vphi,\Delta_j]\nabla\psi\right\|_{L^\infty}\right\|_{\ell^r}d\tau\;\leq\;
C\Int^t_0\left(\left\|\nabla\vphi\right\|_{L^\infty}\,
\left\|\psi\right\|_{B^{s}_{\infty,r}}\,+\,
\left\|\nabla\vphi\right\|_{B^{s-1}_{\infty,r}}
\,\left\|\nabla\psi\right\|_{L^\infty}\right)d\tau.
\end{equation}
\end{prop}

  %%%%%%%%%%%%%%%%%% PARAPRODUCT %%%%%%%%%%%%%%%%%
Let us recall the \emph{Bony's paraproduct decomposition} (first introduced in \cite{Bony}):
\begin{equation}\label{eq:bony}
uv\,=\,T_uv\,+\,T_vu\,+R(u,v)\,,
\end{equation}
where we defined the paraproduct operator $T$ and the remainder $R$ as
$$
T_uv:=\sum_j S_{j-1}u\dj v\ \hbox{ and }\
R(u,v):=\sum_j\sum_{|j'-j|\leq1}\dj u\,\Delta_{j'}v\,.
$$
These operators enjoy the following continuity properties in the class of Besov spaces.
\begin{prop}\label{p:prod}
For any $(s,p,r)\in\R\times[1,\infty]^2$ and $t>0$, the paraproduct operator
$T$ maps $L^\infty\times B^s_{p,r}$ in $B^s_{p,r},$
and  $B^{-t}_{\infty,\infty}\times B^s_{p,r}$ in $B^{s-t}_{p,r}.$
Moreover, the following estimates hold:
$$
\|T_uv\|_{B^s_{p,r}}\leq C\|u\|_{L^\infty}\|\nabla v\|_{B^{s-1}_{p,r}}\quad\hbox{and}\quad
\|T_uv\|_{B^{s-t}_{p,r}}\leq C\|u\|_{B^{-t}_{\infty,\infty}}\|\nabla v\|_{B^{s-1}_{p,r}}.
$$
For any $(s_1,p_1,r_1)$ and $(s_2,p_2,r_2)$ in $\R\times[1,\infty]^2$ such that
$s_1+s_2>0,$ $1/p:=1/p_1+1/p_2\leq1$ and $1/r:=1/r_1+1/r_2\leq1$
the remainder operator $R$ maps
$B^{s_1}_{p_1,r_1}\times B^{s_2}_{p_2,r_2}$ in $B^{s_1+s_2}_{p,r}.$
\end{prop}

%%%%%%%%% TIME-DEPENDENT BESOV SPACES %%%%%%%%%%%%%%%%%%%%%%%
When solving evolutionary PDEs in Besov spaces, we have to localize the equations by Littlewood-Paley decomposition. So we will have
estimates for the Lebesgue norm of  each dyadic block \emph{before}
performing integration in time. This leads to the following  definition, introduced for the first time in paper \cite{Chemin-Lerner}
by Chemin and Lerner.
\begin{defin}\label{def:Besov,tilde}
For  $s\in \R$, $(q,p,r)\in [1,+\infty]^3$ and $T\in [0,+\infty]$, we set
$$
\|u\|_{\tilde L^q_T(B^s_{p,r})}\,:=\,\Bigl\| \Bigl(  2^{js}  \|\dj u(t)\|_{L^q_T(L^p)} \Bigr)_{j\geq -1}\Bigr\|_{\ell^r}\,.
$$
We also set $\tilde C_T(B^s_{p,r})=\tilde L_T^\infty(B^s_{p,r})\cap C([0,T];B^s_{p,r})$.
\end{defin}
The relation between these classes and the classical $L^q_T(B^s_{p,r})$ can be easily recovered by Minkowski's inequality:
$$
\left\{\begin{array}{lcl}
        \|u\|_{\wtilde{L}^q_T(B^s_{p,r})}\;\leq\;\|u\|_{L^q_T(B^s_{p,r})} & \mbox{ if } & q\,\leq\,r \\[1ex]
\|u\|_{\wtilde{L}^q_T(B^s_{p,r})}\;\geq\;\|u\|_{L^q_T(B^s_{p,r})} & \mbox{ if } & q\,\geq\,r\,.
       \end{array}\right.
$$

Combining the above proposition \ref{p:prod} with Bony's decomposition \eqref{eq:bony},
we easily get the following product estimate in Chemin-Lerner space:
\begin{coroll}\label{c:op}
There exists a constant $C$ depending only on $d, s, p, r  $
such that 
$$
\|uv\|_{\tilde L^q_T(B^s_{p,r})}
\leq C\left(\|u\|_{L^{q_1}_T(L^\infty)}\|v\|_{\tilde L^{q_2}_T(B^{s}_{p,r})}
\,+\, \|u\|_{\tilde L^{q_3}_T(B^{s}_{p,r})} \| v\|_{L^{q_4}_T(L^\infty)}\right),
\quad  \frac 1q:=\frac{1}{q_1}  +\frac{1}{q_2}=\frac{1}{q_3} +\frac{1}{q_4}.
$$
\end{coroll}

%%%%%%%%%%%%%%%%%%%% ACTION %%%%%%%%%%%%%%%%%%%
One also has the estimates for the composition of functions in Besov spaces.
\begin{prop}\label{p:comp_grad}
Let  $F:\R\rightarrow\R$ be a smooth function.
Then for any   $s>0$, $(q,p,r)\in[1,+\infty]^3$,   we have
$$
\|\nabla(F(a))\|_{\tilde L^q_T(B^{s-1}_{p,r})}\leq C\|\nabla a\|_{\tilde L^q_T(B^{s-1}_{p,r})}.
$$
If furthermore   $F(0)=0$, then
$   \left\|F(a)\right\|_{\tilde L^q_T(B^s_{p,r})}\,\leq\,C\,\|a\|_{\tilde L^q_T(B^s_{p,r})}\,.
$ 
\end{prop}

\medbreak
In the next section we will need also some notions about \emph{homogeneous paradifferential calculus}: let us recall them.

The homogeneous dyadic blocks $(\dot\Delta_j)_{j\in\Z}$ are defined by
$$
\dot\Delta_j:=\varphi(2^{-j}D)\qquad \text{ if }\quad  j\in \Z.
$$
The homogeneous low frequency cut-off are defined by:
$$
\dot S_ju:=\chi(2^{-j}D)u \qquad \mbox{ for }\quad j\in \Z.
$$
The homogeneous paraproduct operator and remainder operator are defined by:
$$
\dot T_uv:=\sum_{j\in \Z} \dot S_{j-1}u\ddj v\qquad \mbox{ and }\qquad
\dot R(u,v):=\sum_{j\in \Z}\sum_{|j'-j|\leq1}\ddj u\,\dot\Delta_{j'}v.
$$
Notice that, for all $u$ and $v$ in $\cS'$, the sequence $(\dot S_{j-1}u\, \ddj v)_{j\in\Z}$ is spectrally
supported in dyadic annuli. The analogous of Proposition \ref{p:prod} holds true also in the homogeneous setting.

Finally, let us set $\dot C^s=\dot B^s_{\infty,\infty}$, for $s>0$, to be the homogeneous H\"older space.
Recall that, for any $u\in\dot C^s$, the equality $u=\sum_{j\in \Z}\ddj u$ holds.
For homogeneous H\"{o}lder spaces and time-dependent homogeneous H\"{o}lder spaces, we have the following characterization.
\begin{prop}\label{prop:space,equiv}
$\forall \epsilon\in (0,1)$, there exists a constant $C$ such that for all $u\in \cS$,
\begin{equation}\label{equiv:space}
C^{-1} \|u\|_{ \dot C^\epsilon }
\leq \left\|  \frac{\|u(x+y)-u(x)\|_{L^\infty_x}}{|y|^\epsilon}\right\|_{L^\infty_y}
\leq C \|u\|_{ \dot C^\epsilon },
\end{equation}
and
\begin{equation}\label{equiv:space,time}
C^{-1} \|u\|_{\tilde L^1_t(\dot C^\epsilon)}
\leq \left\|\Int^t_0 \frac{\|u(\tau,x+y)-u(\tau,x)\|_{L^\infty_x}}{|y|^\epsilon}\, d\tau\right\|_{L^\infty_y}
\leq C \|u\|_{\tilde L^1_t(\dot C^\epsilon)}.
\end{equation}
\end{prop}
%%%% PROOF %%%%%
\begin{proof} The proof of \eqref{equiv:space} can be found at page 75 of \cite{B-C-D}.

Let us just show the left-hand
inequality of \eqref{equiv:space,time}. Since
$$
\ddj u(t,x)=2^{jd}\Int_{\R^d} h(2^j y) (u(t,x-y) - u(t,x))\, dy,
$$
then we easily find
\begin{align*}
\Int^t_0 2^{j\epsilon} \|\ddj u(\tau,\cdot)\|_{L^\infty_x}
&\leq 2^{jd} \Int_{\R^d} 2^{j\epsilon} |y|^\epsilon | h(2^j y) |
     \Int^t_0 \frac{ \| u(\tau,x-y) - u(\tau,x) \|_{L^\infty_x}}{|y|^\epsilon}\,d\tau\, dy \\
&\leq C\left\|\Int^t_0 \frac{\|u(\tau,x+y)-u(\tau,x)\|_{L^\infty_x}}{|y|^\epsilon}\right\|_{L^\infty_y}.
\end{align*}
This relation gives us the left-hand side inequality  of \eqref{equiv:space,time}.

The inverse inequality follows immediately after similar changes with respect to time in the classical proof.
\end{proof}

Recall the classical a priori estimates for heat equations   in homogeneous Besov spaces:
\begin{prop}\label{prop:holder,est}
For any $s\in \R$, there exists a constant $C_0$ such that
\begin{equation}\label{holder-heat}
\|F\|_{\wtilde L^\infty_T(\dot C^s)\cap \tilde L^1_T(\dot C^{2+s})}
\leq C_0(\|f_0\|_{\dot C^s}+\|f\|_{\tilde L^1_T(\dot C^s)}),
\end{equation}
where $f_0,f,F\in \cS(\R^d)$ and they are linked by the relation
$$F(t,x)=  e^{t \Delta}f_0  + \Int^t_0 e^{(t-\tau)\Delta}f(\tau)\, d\tau.$$
\end{prop}

 %%%%%%%%%%% PROPOSITION PARAPRODUCT HOMOGENEOUS %%%%%%%%%%%%%%%%%%%%%%%Ã¹Ã¹
For later use, it is also convenient to show an a priori estimate for the paraproduct $\dot T_v u$ in the space $\tilde L^1_T(\dot C^s)$.
\begin{lemma}\label{lem:holder-product}
For any $s>0$, $\eps>0$, $a,b>0$,  there exists a constant $C_\eps\sim \eps^{-a/b}$ such that
\begin{equation}\label{est:holder-paraproduct}
\|\dot T_v u\|_{\tilde L^1_T(\dot C^s)}\leq C_\eps \Int^T_0  \|u(t)\|_{\dot C^s}^{\frac{a+b}{b}} \|v\|_{\dot C^{-a}}
                                               +\eps \|v\|_{\tilde L^1_T(\dot C^{b})}.
\end{equation}
\end{lemma}
%%%%%PROOF %%%%%%%%%%%
\begin{proof}
First, let us notice that,
\begin{align*}
\|\dot T_v u\|_{\tilde L^1_T(\dot C^s)}\leq   \Int^T_0 \sum_{j\in \Z} \|u\|_{\dot C^s}\|\ddj v\|_{L^\infty}.
\end{align*}
For any $\eps>0$, $a,b>0$ and $t\in (0,T)$, we fix an integer
$$
N_t=\left[ \frac 1b \log_2 (\eps^{-1}\|u(t)\|_{\dot C^s}) \right]+1,
$$
then  noticing that $\eps^{-1} \|u(t)\|_{\dot C^s}\sim 2^{N_t b}$, we have
\begin{align*}
\Int^T_0 \sum_{j\in \Z} \|u\|_{\dot C^s}\|\ddj v\|_{L^\infty}
&\leq \Int^T_0 \sum_{ j\leq N_t} 2^{ja} \|u(t)\|_{\dot C^s} 2^{-ja} \|\ddj v\|_{L^\infty}
               +\sum_{j\geq N_t+1} 2^{-jb} \|u(t)\|_{\dot C^s} 2^{jb} \|\ddj v\|_{L^\infty}  \\
&\lesssim \Int^T_0  2^{N_t a} \|u(t)\|_{\dot C^s} \|v\|_{\dot C^{-a}}
                +\sum_{j\geq N_t+1} 2^{-(j-N_t)b}\, \eps\,  2^{jb} \|\ddj v\|_{L^\infty} \\
&\lesssim  \Int^T_0 \eps^{-a/b} \|u(t)\|_{\dot C^s}^{(a+b)/b} \|v\|_{\dot C^{-a}}
               + \eps \sum_{j\geq 1} 2^{-jb}\, 2^{(j+N_t)b} \|\dot\Delta_{j+N_t} v\|_{L^\infty} \\
&\lesssim  \eps^{-a/b} \Int^T_0  \|u(t)\|_{\dot C^s}^{(a+b)/b} \|v\|_{\dot C^{-a}}
                +\eps \sup_{j}  \Int^T_0 2^{jb} \|\ddj v\|_{L^\infty}.
\end{align*}
Thus the lemma follows.
\end{proof}

%%%%%%%%%%%%%%%%%%%%%%%%%%%%%%%%%%%%%%%%%%%%%%%%%%%%%%%%%%%%%%%%%%%%%%%%%%%%%%%%%%%%%%%%%%%%%%%%%%%%%%%%%%%%%%%%%%%%%
%%%%%%%%%%%%%%%%%%%%%%%%%%%%%%%%%%%%%%%%%%%%%%%%%%%%%%%%%%%%%%%%%%%%%%%%%%%%%%%%%%%%%%%%%%%%%%%%%%%%%%%%%%%%%%%%%%%%%
\section{A priori estimates for  parabolic and transport equations in endpoint Besov spaces} \label{s:est_Besov}
%%%%%%%%%%%%%%%%%%%%%%%%%%%%%%%%%%%%%%%%%%%%%%%%%%%%%%%%%%%%%%%%%%%%%%%%%%%%%%%%%%%%%%%%%%%%%%%%%%%%%%%%%%%%%%%%%%%%%
%%%%%%%%%%%%%%%%%%%%%%%%%%%%%%%%%%%%%%%%%%%%%%%%%%%%%%%%%%%%%%%%%%%%%%%%%%%%%%%%%%%%%%%%%%%%%%%%%%%%%%%%%%%%%%%%%%%%%

The present section is devoted to obtain new a priori estimates for parabolic
and transport equations in the endpoint Besov spaces.
These estimates will be the key point to get our results.

In the first subsection we will focus on the parabolic equations, and we will prove a priori estimates in Chemin-Lerner spaces based on endpoint Besov spaces 
$B^s_{\infty,r}$. This will be useful in the proof of Theorem \ref{thm:L2wp}.

In the second subsection we will prove refined a priori estimates in the enpoint space $B^0_{\infty,1}$ for linear transport equations. This
is a generalization of such results in \cite{H-K-2008,V}, and it will be fundamental in getting
lower bounds for the lifespan of the solution to the inviscid zero-Mach number system.

%%%%%%%%%%%%%%%%%%%%%%%%%%%%%%%%%%%%%%%%%%%%%%%%%%%%%%%%%%%%%%%%%%%%%%%%%%%%%%%%%%%%%%%%%%%%%%%%%%%%%%%%%%%%%%%%%%%%%%
%%%%%%%%%%%%%%%%%%%%%%%%%%%%%%%%%%%%%%%%%%%%%%%%%%%%%%%%%%%%%%%%%%%%%%%%%%%%%%%%%%%%%%%%%%%%%%%%%%%%%%%%%%%%%%%%%%%%%%
\subsection{Parabolic estimates in $B^s_{\infty,r}$} \label{s:parabolic}
%%%%%%%%%%%%%%%%%%%%%%%%%%%%%%%%%%%%%%%%%%%%%%%%%%%%%%%%%%%%%%%%%%%%%%%%%%%%%%%%%%%%%%%%%%%%%%%%%%%%%%%%%%%%%%%%%%%%%%
%%%%%%%%%%%%%%%%%%%%%%%%%%%%%%%%%%%%%%%%%%%%%%%%%%%%%%%%%%%%%%%%%%%%%%%%%%%%%%%%%%%%%%%%%%%%%%%%%%%%%%%%%%%%%%%%%%%%%%

The present subsection is devoted to state new a priori estimates for linear parabolic equations in endpoint Besov spaces $B^s_{\infty,r}$.
For the reasons explained in the previous section, we will work in the time-dependent Besov spaces
$\tilde L^q_T(B^s_{\infty,r})$ defined above.

\begin{prop}\label{prop:heat-holder}
Let $\rho\in \cS(\R^d)$ solve the following linear parabolic equation
\begin{equation}\label{holder-eq:para}
\left\{
\begin{array}{c}
\d_t\rho-\div(\kappa\nabla\rho)=f,\\
\rho|_{t=0}=\rho_0,
\end{array}
\right.
\end{equation}
with $\kappa,f,\rho_0\in \cS(\R^d)$ and
$$
0<\rho_\ast\leq \rho_0 \leq \rho^\ast,\quad
0<\kappa_\ast\leq \kappa(t,x)\leq \kappa^\ast.
$$
 If $s>0$,  then for any $\epsilon\in (0,1)$, there exists a constant $C$,
depending on $d,s,r,\rho_\ast,\rho^\ast,\kappa_\ast,\kappa^\ast$, such that  the following estimate holds true:
\begin{eqnarray*}\label{heat-holder:rho,s}
& & \hspace{-0.3cm}\|\rho\|_{\wtilde L^\infty_t(B^s_{\infty,r})\cap \wtilde L^1_t(B^{s+2}_{\infty,r})}
\,\leq\,C \,\left[1+\|\kappa\|_{L^\infty_t(B^\epsilon_{\infty,\infty})}^{2/\epsilon}\right]\,
\times \Biggl(\|\rho_0\|_{B^s_{\infty,r}} + \|f\|_{\wtilde L^1_t(B^s_{\infty,r})}
         \\
& &\qquad\qquad\qquad
+\int^t_0\left( \,\left[1+\|\kappa\|_{B^{1+\epsilon}_{\infty,\infty}}^{2/(1+\epsilon)}\right]\,
         \|\rho\|_{B^s_{\infty,r}}+\|\nabla\kappa\|_{L^\infty}\|\nabla\rho\|_{B^s_{\infty,r}}
         +\|\nabla\kappa\|_{B^s_{\infty,r}} \|\nabla\rho\|_{L^\infty}\right)d\tau\Biggr).
\end{eqnarray*}
\end{prop}

\begin{rem}\label{rem:rho}
By Gronwall lemma, it is easy to get the following a priori estimate for $\rho$:
\begin{align*}
\|\rho\|_{\wtilde L^\infty_t(B^s_{\infty,r})\cap  \wtilde L^1_t(B^{s+2}_{\infty,r})}\,
 \leq\,
      C_1\, \exp\{ C_1 K(t)  \}\,
        \bigl ( \|\rho_0\|_{B^s_{\infty,r}} + \|f\|_{\wtilde L^1_t(B^s_{\infty,r})} \bigr)\,,
\end{align*}
where $C_1=C_1(t)$ depends on $d,s,r,\rho_\ast,\rho^\ast,\kappa_\ast,\kappa^\ast$
and $\|\kappa\|_{L^\infty_t(B^\epsilon_{\infty,\infty})}$, and  
$$
K (t):=\Int^t_0 \left( 1+ \| \nabla\kappa\|_{ L^\infty}^{2 }
    +\|\nabla\kappa\|_{B^{s}_{\infty,r}}^{\max\{2/(1+s),1\}} \right) \,d\tau\,.
$$
\end{rem}

Proposition \ref{prop:heat-holder} will be proved in three steps. The strategy is the following.

First of all, we localize the function $\rho$ into countable functions $\vrho_n$, each of which is supported on some ball $B(x_n,\delta)$,
with small radius $\delta\in (0,1)$ to be determined in the proof. 
Hence up to some small perturbation, $\vrho_n$ verifies a heat equation with a \emph{time-dependent} heat-conduction coefficient
(see System \eqref{holder-eq:para-n} below).
Consequently, changing the time variable and making use of estimates for the heat equation entail a control for $\vrho_n$ in  $ \dot E^\epsilon$:
 $$
  \dot E^\epsilon:=  \tilde L^\infty_T( \dot C^\epsilon)\cap \tilde L^1_T (\dot  C^{2+\epsilon}),
  \quad  \epsilon\in (0,1),
$$
This will be done in Step 1.

In step 2, thanks to Proposition \ref{prop:space,equiv}, we carry the result from $\{\vrho_n\}$ to $\rho$.
Note that Maximum Principle applied to parabolic equations has already given us the control on low frequencies of the solution:
\begin{equation}\label{holder:bound}
\|\rho\|_{L^\infty_t(L^\infty)}\leq \|\rho_0\|_{L^\infty}+\Int^t_0 \|f\|_{L^\infty}.
\end{equation}
This already ensures the result in H\"older space $C^\epsilon$.

Step 3 is devoted to  handle general Besov spaces of form $B^s_{\infty,r}$: we again localize the system, but in \textit{Fourier variables};
then we apply the result of Step 2 to $\rho_j:=\dj \rho$ and a careful calculation on commutator terms will yield the thesis.

\medbreak
We agree that in this subsection $\{\varrho_n(t,x)\}$ always denote localized functions of $\rho(t,x)$ in \textit{$x$-space},
while $\rho_j$ as usual, denotes $\dj\rho$ (localization in the \textit{phase space}).

%%%%%%%%% MAIN PROOF %%%%%%%%%%%%%%%%%%%%%%%%%%%%%%%%%%%%%%%%%%%%%%%%%%%%%%%%%%%%%
%%%%%%%%%%%%%%%%%%%%%%%%%%%%%%%%%%%%%%%%%%%%%%%%%%%%%%%%%%%%%%%%%%%%%%%%%%%%%%%%%%%%%%%%%%%%%%%%
%%%%%%%%%%%%%%%%%%%%%%%%%%%%%%%%%%%%%%%%%%%%%%%%%%%%%%%%%%%%%%%%%%%%%%%%%%%%%%%%%%%%%%%

\subsubsection*{Step 1: estimate for $\vrho_n$ in $\dot E^\epsilon$}%%%%%%%%%%%%%%%%%%%%%% ESTIMATE %%%%%%%%%%%%%%%%%%%%%%%%%%%%%%%%%%%%%%%%%%%%%%%%%%%%%%%%%%%%%%%%%%%%%%%%%%%%%%%%%%%%%%

  Let us take first a smooth partition of unity $\{\psi_n\}_{n\in \N}$ subordinated to a locally finite covering of $\R^d$.
We suppose that the $\psi_n$'s satisfy the following conditions:
\begin{enumerate}[(i)]
\item $\Supp \psi_n\subset B(x_n,\delta)\triangleq B_n$, with $\delta<1$ to be determined later;
\item $\sum_{n}\psi_n\equiv 1$;
\item $0\leq \psi_n\leq 1$, with $\psi_n\equiv 1$ on $B(x_n,\delta/2)$;
\item $\|\nabla^\eta \psi_n\|_{L^\infty}\leq C|\delta|^{-|\eta|}$, for $|\eta|\leq 3$;
\item for each $x\in \R^d$, there are at most  $N_d$ (depending on the dimension $d$)
elements in  $\{\psi_n\}_{n\in \N}$ covering the ball $B(x,\delta/2)$.
\end{enumerate}

%%%%%%%%%%%%%%%%%%%%%%% EQUATION FOR VRHO_n%%%%%%%%%%

Now by multiplying $\psi_n$ to Equation \eqref{holder-eq:para}, we get the equation for $\vrho_n\triangleq \rho\psi_n$,
which is compactly supported on $B_n$:
\begin{equation}\label{holder-eq:para-n}
\left\{
\begin{array}{c}
\d_t\vrho_n-\bar\kappa_n\Delta\vrho_n=(\kappa-\bar\kappa_n)\Delta\vrho_n+\nabla\kappa\cdot\nabla\vrho_n+g_n,\\
\vrho_n|_{t=0}=\vrho_{0,n}=\psi_n \rho_0,
\end{array}
\right.
\end{equation}
where
$$
\bar\kappa_n(t)\triangleq \frac{1}{\textrm{vol}(B_n)}\Int_{B_n}\kappa(t,y)\,dy
$$
is a function depending only on $t$, and
\begin{equation}\label{g_n}
g_n=-2\kappa\nabla\psi_n\cdot\nabla\rho-(\kappa\Delta\psi_n+\nabla\kappa\cdot\nabla\psi_n)\rho+f\psi_n.
\end{equation}
For convenience  we suppose that there exists a  positive constant $C_\kappa $
such that
\begin{equation}\label{holder-kappa}
\quad |\kappa(t,x)-\kappa(t,y)|\leq C_0 \|\kappa\|_{L^\infty_{t_0}(\dot C^\epsilon)} |x-y|^\epsilon,\quad\forall\, x,y\in \R^d,\, t\in [0,t_0].
\end{equation}
Notice that, by \eqref{holder-kappa}, we have $\bar \kappa_n\geq \kappa_\ast>0$, which ensures that, for all $t\in [0,t_0]$,
\begin{equation}\label{holder-small}
\| \kappa/\bar\kappa_n-1  \|_{L^\infty(B_n)}
\leq \kappa_\ast^{-1}\left\|\frac{1}{\textrm{vol}(B_n)}\Int_{B_n}\kappa(t,x)-\kappa(t,y)\,dy\right\|_{L^\infty(B_n)}
\leq C_\kappa\|\kappa\|_{L^\infty_{t_0}(\dot C^\epsilon)} \kappa_\ast^{-1}\delta^\epsilon.
\end{equation}

%%%% TIME TRANSFORMATION %%%%

In order to get rid of  the variable coefficient $\bar \kappa_n(t)$, let us make the one-to-one change in time variable
\begin{equation}\label{holder-transform}
\tau \triangleq  \tau(t) = \Int^t_0 \bar\kappa_n(t')\, dt'.
\end{equation}
Therefore, the new unknown
$$
\tilde \varrho_n(\tau,x)\triangleq
%%\tilde \vrho_n(t,x) \hbox{ with }\tilde \rho_n(\tau,x)
 \rho_n(t,x),
$$
 satisfies (observe that $\frac{d\tau}{ dt}=\bar\kappa_n(t)$)
\begin{equation}\label{holder-eq:para-tilde}
\left\{
\begin{array}{c}
\d_\tau\tilde\varrho_n-\Delta\tilde \varrho_n
=\left(\frac{\tilde\kappa(\tau)}{\tilde\kappa_n(\tau)}-1\right) \Delta\tilde\varrho_n
  +\frac{\nabla\tilde \kappa(\tau)}{\tilde\kappa_n(\tau)}\cdot\nabla\tilde \varrho_n
  +\frac{\tilde g_n(\tau)}{\tilde\kappa_n(\tau)}
 % + \tilde \Theta_n(\tau,x)
  ,\\
\tilde\varrho_n|_{\tau=0}=\rho_{0,n},
\end{array}
\right.
\end{equation}
where $\tilde\kappa(\tau,x)=\kappa(t,x),\, \tilde\kappa_n(\tau)=\bar\kappa_n(t),\, \tilde \rho(\tau,x)=\rho(t,x),\, \tilde g_n(\tau,x)=g_n(t,x)$.

 %%%%%%%%%%%%%%%ESTIMATE FOR SOURCE TERM 1 %%%%

This is a heat equation: in view of Proposition \ref{prop:holder,est}, we have to bound the ``source'' terms.
Estimate \eqref{est:holder-paraproduct} and
\begin{equation*}
\|\dot T_u v + \dot R(u,v)\|_{\tilde L^1_T(\dot C^\epsilon)}\leq C\|u\|_{L^\infty_T(L^\infty)} \|v\|_{\tilde L^1_T(\dot C^\epsilon)}
\end{equation*}
imply that the first source term of Equation \eqref{holder-eq:para-tilde} can be controlled by
\begin{align*}
\left\|\left(\frac{\tilde\kappa(\tau,\cdot)}{\tilde\kappa_n(\tau)}-1\right)
     \Delta\tilde\varrho_n(\tau,\cdot)\right\|_{\tilde L^1_T(\dot C^\epsilon)}
&\leq C\left\|\tilde\kappa/\tilde\kappa_n-1 \right\|_{L^\infty_T(L^\infty(B_n))}\,
                    \|\Delta\tilde \varrho_n\|_{\tilde L^1_T(\dot C^\epsilon)}\\
&\quad +C_{\eta_1} \Int^T_0 \left\|\tilde\kappa/\tilde\kappa_n-1 \right\|_{\dot C^\epsilon}^{\frac{2}{\epsilon}}
                          \|\Delta\tilde \varrho_n\|_{\dot C^{\epsilon-2}}
                 +\eta_1 \|\Delta\tilde \varrho_n\|_{\tilde L^1_T(\dot C^\epsilon)},
\end{align*}
for any   $\eta_1\in (0,1)$ with $C_{\eta_1}\sim \eta_1^{ \frac{\epsilon-2}{\epsilon}}$. Besides, Inequality \eqref{holder-small} ensures that
for all $\tau\in [0,\tau_0]$, with $\tau_0=\tau(t_0)$,
$$
\|\tilde\kappa/\tilde\kappa_n-1  \|_{L^\infty_{\tau_0}(B_n)}
\leq C_\kappa\|\kappa\|_{L^\infty_{t_0}(\dot C^\epsilon)} \kappa_\ast^{-1}\delta^\epsilon,
$$
which implies, for $T\in [0,\tau_0]$,
\begin{equation}\label{holder-est:source1}
\left\|\left(\frac{\tilde\kappa(\tau,\cdot)}{\tilde\kappa_n(\tau)}-1\right)
\Delta\tilde\varrho_n(\tau,\cdot)\right\|_{\tilde L^1_T(\dot C^\epsilon)}
\leq C_{\eta_1} \Int^T_0  \left\|\tilde\kappa \right \|_{\dot C^\epsilon}^{\frac{2}{\epsilon}}
                          \| \tilde \varrho_n\|_{\dot C^{\epsilon}}
       +(CC_\kappa \|\kappa\|_{L^\infty_{t_0}(\dot C^\epsilon)} \kappa_\ast^{-1} \delta^\epsilon+\eta_1)
       \|\tilde \varrho_n\|_{\tilde L^1_T(\dot C^{2+\epsilon})}.
\end{equation}

 %%%%%%%%%%%%% ESTIMATE FOR SOURCE TERM 2%%%%%
For any $\eta>0$, there exists $C_{\eta}\sim \eta^{-1}$ such that
$$
\|T_u v + R(u,v)\|_{\tilde L^1_T(\dot C^\epsilon)}
\leq C_{\eta}\Int^T_0 \|u\|_{L^\infty}^{2} \|v\|_{\dot C^{\epsilon-1}}
                                                +\eta \|v\|_{\tilde L^1_T(\dot C^{\epsilon+1})}.
$$
Thus, also by use of Lemma \ref{lem:holder-product} with $a=1-\epsilon$ and $b=1+\epsilon$, for any $\eta_2\in (0,1)$ we have the following (with $C_{\eta_2}\sim\eta_2^{-1}$):
\begin{align}\label{holder-est:source2}
\left \|\frac{\nabla\tilde \kappa(\tau, \cdot)}{\tilde\kappa_n(\tau)}
            \cdot\nabla\tilde \varrho_n(\tau, \cdot) \right \|_{\tilde L^1_T(\dot C^\epsilon)}
\leq C_{\eta_2} \Int^T_0
              \left (\|\nabla\tilde\kappa\|_{L^\infty}^{2}
                   + \|\nabla\tilde\kappa\|_{\dot C^\epsilon}^{\frac{2}{1+\epsilon}} \right)
                  \|\tilde\varrho_n\|_{\dot C^{\epsilon}}
                  +\eta_2 \|\tilde \varrho_n\|_{\tilde L^1_T(\dot C^{2+\epsilon})}.
\end{align}

%%%%%%%%%%%%ESTIMATE FOR TILDE RHO_n

Now let us choose $\delta,\eta_1,\eta_2$ such that
\begin{equation}\label{heat-parameter}
C_0CC_\kappa \|\kappa\|_{L^\infty_{t_0}(\dot C^\epsilon)}  \kappa_\ast^{-1} \delta^\epsilon\;,\;\; C_0\eta_1\;,\;\; C_0\eta_2\;\leq\; 1/6,
\end{equation}
with the same $C_0$ in \eqref{holder-heat}.  Then, from Proposition \ref{prop:holder,est} and
estimates \eqref{holder-est:source1}, \eqref{holder-est:source2}, for any $t\in[0,T_0]$ we get, for some ``harmless''
constant still denoted by $C$,
$$
\|\tilde \varrho_n\|_{L^\infty_T(\dot C^\epsilon)\cap \tilde L^1_T(\dot C^{2+\epsilon})}
 \leq C \left( \|\vrho_{0,n}\|_{\dot C^\epsilon}    +   \Int^T_0
      \left(\|\tilde \kappa\|_{\dot C^\epsilon}^{\frac{2 }{\epsilon}}\,
         +\, \|\nabla\tilde\kappa\|_{L^\infty}^{2 }
       \,+\, \|\nabla\tilde \kappa\|_{\dot C^\epsilon}^{\frac{2 }{1+\epsilon}}\right)
          \|\tilde \varrho_n\|_{\dot C^{\epsilon }}
       + \|\tilde g_n\|_{\tilde L^1_T(\dot C^\epsilon)}\right).
$$
 Since $\kappa_\ast\leq \bar\kappa_n\leq \kappa^\ast$, after transformation in time \eqref{holder-transform} we arrive at
\begin{equation}\label{holder-est:rho-n}
\| \vrho_n\|_{L^\infty_T(\dot C^\epsilon)\cap \tilde  L^1_T( \dot C^{2+\epsilon})}
\leq C  \left( \|\vrho_{0,n}\|_{\dot C^\epsilon}+
  \Int^T_0 K_1\|\vrho_n\|_{\dot C^{\epsilon}}
                               + \|  g_n\|_{\tilde L^1_T(\dot C^\epsilon)} \right)
\end{equation}
for all $T\in [0,t_0]$, with
$$
K_1=\|\kappa\|_{\dot C^\epsilon}^{\frac{2 }{\epsilon}}+\|\nabla\kappa\|_{L^\infty}^{2 }
       + \|\nabla\kappa\|_{\dot C^\epsilon}^{\frac{2 }{1+\epsilon}},
$$
provided that we choose
\begin{equation}\label{delta}
\delta^{-\epsilon}\,=\,1\,+\,\wtilde{C}\,\|\kappa\|_{L^\infty_{t_0}(\dot C^\epsilon)}
\qquad\mbox{for some  constant } \wtilde{C}\,\mbox{ depending only on } d\,,\,\epsilon\,.
\end{equation}

\subsubsection*{Step 2: H\"older estimates for $\rho$} %%%%%%%% Rho %%%%%%%%%%%%%%%%%%%%%%%%%%%%%%%%%%%%%%%
%%%%%%%%%%%%%%%%%%%%%%%%%%%%%%%%%%%%%%%%%%%%%%%%%%%%%%%%%%%%%%%%%%%%%%%%%%%%%%%%%%%%%%%%%%%%%%%%%%

Now we come back to consider $\rho=\sum_n \varrho_n$.
By assumptions on the partition of unity $\{\psi_n\}$, for any $x$ there exist
$N_d$ balls of our covering which cover the small ball $B(x,\delta/4)$.
Therefore, from inequality \eqref{equiv:space} we have
\begin{align*}
\|\rho\|_{\tilde L^1_t(\dot C^\epsilon)}
&\leq C \left\|\Int^t_0 \frac{\|\rho (\tau, x+y)-\rho(\tau, x)\|_{L^\infty_x}}{|y|^\epsilon}\right\|_{L^\infty_y}\\
&\leq C\sup_{|y|>\delta/4}\Int^t_0 \frac{\|\rho (x+y)-\rho(x)\|_{L^\infty_x}}{|y|^\epsilon}
    +C\sup_{|y|\leq \delta/4} \Int^t_0 \frac{\|\rho (x+y)-\rho(x)\|_{L^\infty_x}}{|y|^\epsilon},
\end{align*}
whose second term can be controlled by
$$
N_dC\sup_{|y-z|\leq \delta/4} \Int^t_0 \frac{\sup_n \|\vrho_n (x+y)-\vrho_n(x+z)\|_{L^\infty_x}}{|y-z|^\epsilon}.
$$
Thus we find
$$
\|\rho\|_{\tilde L^1_t(\dot C^\epsilon)}
\leq C\delta^{-\epsilon} \Int^t_0 \|\rho\|_{L^\infty}
     +N_d C \sup_n \|\vrho_n\|_{\tilde L^1_t(\dot C^\epsilon)}.
$$
Similarly, we have
$$
\|\rho\|_{\tilde L^\infty_t(\dot C^\epsilon)}
\leq C\delta^{-\epsilon}   \|\rho\|_{L^\infty_t(L^\infty)}
     +C \sup_n \|\vrho_n\|_{\tilde L^\infty_t(\dot C^\epsilon)}.
$$
Since $\nabla^2 \rho=\sum_n (\nabla^2 \vrho_n)$, from the same arguments as before we infer
$$
\|\rho\|_{\tilde L^1_t(\dot C^{2+\epsilon})}
\leq  C\|\nabla^2\rho\|_{\tilde L^1_t(\dot C^\epsilon)}
\leq C\delta^{-\epsilon}  \Int^t_0 \|\nabla^2\rho\|_{L^\infty}
     +C \sup_n \|\vrho_n\|_{\tilde L^1_t(\dot C^{2+\epsilon})}.
$$
Therefore, to sum up, for all $t\in [0,t_0]$,
\begin{align*}
\|\rho\|_{\tilde L^\infty_t(\dot C^\epsilon)\cap \tilde L^1_t(\dot C^{2+\epsilon})}
&\leq C\delta^{-\epsilon} \left(\|\rho\|_{L^\infty_t(L^\infty)}+ \Int^t_0 \|\nabla^2\rho\|_{L^\infty} \right)
     +C \sup_n \|\vrho_n\|_{\tilde L^\infty_t(\dot C^\epsilon)\cap \tilde L^1_t(\dot C^{2+\epsilon})} \\
&\leq C\delta^{-\epsilon} \left(\|\rho\|_{L^\infty_t(L^\infty)}+ \Int^t_0 \|\nabla^2\rho\|_{L^\infty} \right)+ \\
&\quad  +C \sup_n \left( \|\vrho_{0,n}\|_{\dot C^\epsilon}+
   \Int^t_0 K_1\|\vrho_n\|_{\dot C^{\epsilon}}
                               + \|  g_n\|_{\tilde L^1_t(\dot C^\epsilon)} \right) ,
\end{align*}
with the second inequality deriving from Estimate   \eqref{holder-est:rho-n}. Thanks to \eqref{holder:bound} and the fact that
$$
\|\vrho_n\|_{  C^\epsilon}
=\|\rho \psi_n\|_{  C^\epsilon}\leq C\|\rho\|_{C^\epsilon}\|\psi_n\|_{C^\epsilon}
\leq C\delta^{-\epsilon}  \|\rho\|_{C^\epsilon},
$$
 we thus have the following estimate for $\rho$ in the nonhomogeneous H\"{o}lder space
 \begin{align*}
 \|\rho\|_{\tilde L^\infty_t(  C^\epsilon)\cap \tilde L^1_t(  C^{2+\epsilon})}
&\leq
      C\|\rho\|_{L^\infty_t(L^\infty)}
      +C\Int^t_0 \|\rho\|_{L^\infty}
      +\|\rho\|_{\tilde L^\infty_t(\dot C^\epsilon)\cap \tilde L^1_t(\dot C^{2+\epsilon})}\\
&\leq
      C \delta^{-\epsilon}  \left(\|\rho_0\|_{ C^\epsilon}
      + \Int^t_0 (\|\nabla^2\rho\|_{L^\infty}+\|\rho\|_{L^\infty}+\|f\|_{L^\infty})
      +\Int^t_0 K_1\|\rho \|_{ C^{\epsilon}} \right) \\
&\quad
      +C \sup_n   \|  g_n\|_{\tilde L^1_t(\dot C^\epsilon)}   ,
\end{align*}

 %%%%%%%%%%% g_n %%%%%%%%%%%%%%%%

It rests us to bound $g_n$ uniformly. In fact, starting from definition \eqref{g_n} of $g_n$, we follow the same method
to get \eqref{holder-est:source1} and \eqref{holder-est:source2} and we arrive at
\begin{align*}
\|g_n\|_{\tilde L^1_t( \dot  C^\epsilon)}
&\leq C\Int^t_0
     \left( \eta^{-1}(\|\kappa\nabla\psi_n\|_{L^\infty}^2
        +\|\kappa\nabla\psi_n\|_{\dot C^\epsilon}^{\frac{2}{1+\epsilon}})
           +\|\kappa\Delta\psi_n+\nabla\kappa\cdot\nabla\psi_n\|_{L^\infty} \right)\|\rho\|_{\dot C^\epsilon}\\
&\quad
   +\eta \|\rho\|_{\tilde L^1_t(\dot C^{2+\epsilon})}
   +C\Int^t_0 \|\kappa\Delta\psi_n+\nabla\kappa\cdot\nabla\psi_n\|_{\dot C^\epsilon}  \|\rho\|_{L^\infty}
 +C\delta^{-\epsilon}\Int^t_0 \|f\|_{L^\infty}
 +C\|f\|_{\tilde L^1_t(\dot C^\epsilon)} \\
&\leq C_\eta \delta^{-2} \Int^t_0  \left( 1
           + \|\kappa  \|_{  C^\epsilon}^{\frac{2}{1+\epsilon}}
           + \|\kappa \|_{C^{1+\epsilon}} \right) \|\rho\|_{  C^\epsilon}
 +\eta\|\rho\|_{\tilde L^1_t(\dot C^{2+\epsilon})}
        +C\delta^{-\epsilon}\|f\|_{\tilde L^1_t(  C^\epsilon)},
\end{align*}
where we have used $\|f\|_{L^1_t(L^\infty)\cap \tilde L^1_t(\dot C^\epsilon)}\leq C\|f\|_{\tilde L^1_t(C^\epsilon)}$.

 %%%%%% CONCLUSION %%%%%%%%%%%%

We finally get a priori estimates for $\rho$:
\begin{align*}
\|\rho\|_{\tilde L^\infty_t(  C^\epsilon)\cap \tilde L^1_t(  C^{2+\epsilon})}
&\leq C \delta^{-\epsilon}  \left(   \|\rho_0\|_{  C^\epsilon}
              + \Int^t_0 \| \rho\|_{C^2}
               +\|f\|_{\tilde L^1_t(  C^\epsilon)}   \right)
               +C\delta^{-2} \Int^t_0   K_2 \|\rho\|_{  C^{\epsilon}} , \label{holder-est:rho-dot}
\end{align*}
with
\begin{equation}\label{K2}
K_2 =1 + \| \kappa\|_{  C^{1+\epsilon}}^{\frac{2 }{1+\epsilon}}
    \geq C\left(K_1+1+\|\kappa  \|_{  C^\epsilon}^{\frac{2}{1+\epsilon}}
           +\|\kappa\|_{C^{1+\epsilon}} \right)      .
\end{equation}
Thus, by a direct interpolation inequality, that is to say
$$
\delta^{-\epsilon}\|\rho\|_{L^1_t(C^2)}
\leq C_\eta \delta^{-2}\Int^t_0 \|\rho\|_{C^\epsilon}+\eta \|\rho\|_{\tilde L^1_t(C^{2+\epsilon})},
$$
Gronwall's Inequality tells us
\begin{equation}\label{heat-Holder:rho}
\|\rho\|_{\tilde L^\infty_t(  C^\epsilon)\cap \tilde L^1_t(  C^{2+\epsilon})}
\leq C\delta^{-\epsilon}
\left(\|\rho_0\|_{  C^\epsilon}
               +\|f\|_{\tilde L^1_t(  C^\epsilon)}        \right)
               + C\delta^{-2} \Int^t_0   K_2 \|\rho\|_{  C^{\epsilon}}.
\end{equation}

\subsubsection*{Step 3: the general case $B^s_{\infty,r}$} %%%%%%%% Besov %%%%%%%%%%%%%%%%%%%%%%%%%%%%%%%%%%%%%%%
%%%%%%%%%%%%%%%%%%%%%%%%%%%%%%%%%%%%%%%%%%%%%%%%%%%%%%%%%%%%%%%%%%%%%%%%%%%%%%%%%%%%%%%%%%%%%%%%%%
Now we want to deal with the general case $B^s_{\infty,r}$.  Let us apply $\tilde\dj=\Delta_{j-1}+\dj+\Delta_{j+1}$, $j\geq 0$,
 to System \eqref{holder-eq:para}, yielding
\begin{equation}\label{holder-eq:dj}
\left\{
\begin{array}{c}
\d_t\ov\rho_j-\div(\kappa\nabla\ov\rho_j)=\ov f_j-\ov R_j,\\
\ov\rho_j|_{t=0}=\ov\rho_{0,j},
\end{array}
\right.
\end{equation}
with
$$
\ov\rho_j=\tilde\dj \rho,\quad \ov f_j=\tilde\dj f,
\quad \ov R_j=\div([\kappa,\tilde\dj]\nabla\rho),\quad \ov\rho_{0,j}=\tilde\dj \rho_0.
$$
We apply the a priori estimate \eqref{heat-Holder:rho} to the solution $\ov\rho_j$ of System \eqref{holder-eq:dj}, for some positive
$\epsilon< \min\{s,1\}$, entailing
$$
\|\ov\rho_j\|_{\tilde L^\infty_t(  C^\epsilon)\cap \tilde L^1_t(  C^{2+\epsilon})}
\leq
  C \delta^{-\epsilon}
     \left (\|\ov\rho_{0,j}\|_{C^\epsilon}
   +\|\ov f_j- \ov R_j\|_{\tilde L^1_t(  C^\epsilon)} \right)
     + C\delta^{-2} \Int^t_0   K_2 \|\ov\rho_j\|_{  C^{\epsilon}}.
 $$

Let us notice that for $j\geq 0$, denoted by $\rho_j=\dj \rho$ and $\rho_q=\dq \rho$ as usual, then we have
$$
\dj \ov \rho_j=\rho_j
%\quad  \dq\ov\rho_j =  (\dq+\dj)\rho_q
%\hbox{ if }|q-j|\leq 1
\quad\hbox{and}\quad
\dq \ov\rho_j\equiv 0\hbox{ if }|q-j|\geq 2.
$$
  Hence, due to the dyadic characterization of H\"older spaces, the above inequality gives
$$\displaylines{
2^{j\epsilon}\| \rho_j\|_{  L^\infty_t( L^\infty) }
+2^{j(2+\epsilon)} \Int^t_0 \|\rho_j\|_{L^\infty}\,\leq
\hfill\cr\hfill
\leq
   C \delta^{-\epsilon}
 \left( 2^{j\epsilon} \sum_{ |j-q|\leq 1}
      \Bigl( \| \rho_{0,q}\|_{L^\infty}+\Int^t_0\|  f_q\|_{L^\infty}\Bigr)
      +\|\ov R_j\|_{\tilde L^1_t(  C^\epsilon)} \right)
       + C\delta^{-2} 2^{j\epsilon} \sum_{ |j-q|\leq 1} \Int^t_0   K_2 \| \rho_q\|_{ L^\infty}.
 }$$
Finally, by use also of classical commutator estimates Proposition \ref{p:comm} to control the $\ov R_j$ term,
one gets a priori estimate for $\rho_j$:
$$\displaylines{
 \| \rho_j\|_{  L^\infty_t( L^\infty) }
+2^{ 2j } \Int^t_0 \|\rho_j\|_{L^\infty}
\leq
C\delta^{-2}   \sum_{ |j-q|\leq 1} \Int^t_0   K_2 \| \rho_q\|_{ L^\infty}
\hfill\cr
+ C \delta^{-\epsilon}    \left(
          \sum_{ |j-q|\leq 2}
      \Bigl( \| \rho_{0,q}\|_{L^\infty}+\Int^t_0\|  f_q\|_{L^\infty}\Bigr)
      +2^{-js}c_j
        \Int^t_0 \|\nabla\kappa\|_{B^{s}_{\infty,r}}\| \nabla \rho\|_{L^\infty}+
\Int^t_0 \|\nabla\kappa\|_{L^\infty} \|\nabla\rho\|_{B^s_{\infty,r}}
        \right),
 }$$
for some suitable $\left(c_j\right)_j\,\in\,\ell^r$.

Therefore, we multiply both sides by $2^{js}$ (for $s>-1$) and then take $\ell^r$ norm, to arrive at
$$\displaylines{
\|\rho\|_{\tilde L^\infty_t(B^s_{\infty,r})\cap \tilde L^1_t(B^{s+2}_{\infty,r})}
\leq
\hfill\cr
   \leq C \delta^{-2} \left(
         \|\rho_0\|_{B^s_{\infty,r}} + \|f\|_{\tilde L^1_t(B^s_{\infty,r})}
          +\Int^t_0   K_2 \| \rho \|_{B^s_{\infty,r}}
         +\Int^t_0 \|\nabla\kappa\|_{L^\infty} \|\nabla\rho\|_{B^s_{\infty,r}}
                +\|\nabla\kappa\|_{B^{s}_{\infty,r}}\| \nabla \rho\|_{L^\infty}
      \right).
}$$
The definitions \eqref{delta} of $\delta$ and \eqref{K2} of $K_2$ imply the result. This concludes the proof
of Proposition \ref{prop:heat-holder}.

%%%%%%%%%%%%%%%%%%%%%%%%%%%%%%%%%%%%%%%%%%%%%%%%
\subsection{Transport equations in $B^0_{\infty,1}$} \label{ss:transport}
%%%%%%%%%%%%%%%%%%%%%%%%%%%%%%%%%%%%%%%%%%%%%%%%

We state and prove here new a priori estimates for transport equations
\begin{equation}\label{eq:transport}\left\{\begin{array}{c}
\d_t\omega+v\cdot\nabla\omega=g,\\
\omega|_{t=0}=\omega_0.
\end{array}
\right.\end{equation}
in the enpoint Besov space $B^0_{\infty,1}$.

First of all, let us recall the classical result in the setting of $B^s_{\infty,r}$ classes (see e.g. \cite{B-C-D}, Chapter 3).
\begin{prop}\label{p:transport}  Let $1\leq r\leq\infty$ and
$\sigma>0$ ($\sigma>-1$ if $\div v=0$).
Let $\omega_0\in B^\s_{\infty,r},$
 $g\in L^1([0,T];B^\s_{\infty,r})$ and  $v$
 be a time dependent vector field in $\cC_b([0,T]\times\R^N)$
such that
$$
\begin{array}{lllll}
\nabla v&\in&  L^1([0,T];
L^\infty)&\hbox{\rm if}&\sigma<1,\\[1.5ex]
\nabla v&\in& L^1([0,T];B^{\sigma -1}_{\infty,r})
&\hbox{\rm if}&\sigma>1,\quad
\hbox{\rm  or }\  \sigma=r=1.
\end{array}
$$
Then equation \eqref{eq:transport} has a unique solution  $\omega$ in
\begin{itemize}
\item the space $\cC([0,T];B^{\sigma}_{\infty,r})$ if $r<\infty,$
\item the space $\Bigl(\bigcap_{\sigma'<\sigma} \cC([0,T];B^{\sigma'}_{\infty,\infty})\Bigr)
\bigcap  \cC_w([0,T];B^{\sigma}_{\infty,\infty})$ if $r=\infty.$
\end{itemize}
Moreover,  for all $t\in[0,T],$ we have
\begin{equation}\label{sanspertes1}
e^{-CV(t)}\|\omega(t)\|_{B^\sigma_{\infty,r}}\leq
\|\omega_0\|_{B^\sigma_{\infty,r}}+\int_0^t
e^{-CV(t')}
\|g(t')\|_{B^\sigma_{\infty,r}}\,dt'
\end{equation}
$$\displaylines{
\mbox{with}\quad\!\! V'(t):=\left\{
\begin{array}{l}
\!\|\nabla v(t)\|_{L^\infty}\!\!\!\quad\mbox{if}\!\!\!
\quad\sigma<1,\\[1.5ex]
\!\|\nabla v(t)\|_{B^{\sigma-1}_{\infty,r}}\ \mbox{ if }\
\sigma>1,\quad\!\mbox{or }\
\sigma=r=1.
\end{array}\right.\hfill} $$
 If $\omega=v$ then, for all $\sigma>0$ ($\sigma>-1$ if $\div v=0$),
    Estimate  \eqref{sanspertes1} holds with
$V'(t):=\|\nabla \omega(t)\|_{L^\infty}.$
\end{prop}

Then, the Besov norm of the solution grows in an exponential way with respect to the norm of the transport field $v$.
{}
Nevertheless,  if $v$ is \emph{divernge-free}
then the $B^0_{\infty,r}$ norm of $\omega$ grows linearly in $v$. This
was proved first by Vishik in \cite{V}, and then by Hmidi and Keraani in \cite{H-K-2008}.
{}
Here we generalize their result to the case when $v$ is not divergence free. Of course, we will get a growth also
on $\div v$, which is still suitable for our scopes (see subsection \ref{ss:life}).

\begin{prop}\label{p:transp_0}
Let us consider the linear transport equation \eqref{eq:transport}.

For any $\beta>0$, there exists a constant $C$, depending only on $d$ and $\beta$, such that the following a priori estimate  holds true:
\begin{equation*}\label{est:vort}
\|\omega(t)\|_{B^0_{\infty,1}}\leq
C \Bigl(\|  \omega_0\|_{B^0_{\infty,1}}\,+\,  \|g\|_{L^1_t(B^0_{\infty,1})} \Bigr)
 \Bigl(1+ \cV(t)\Bigr),
\end{equation*}
with $\cV(t)$ defined by
$$\cV(t):=  \Int^t_0\|\nabla v\|_{L^\infty}\,+\, \|\div v\|_{B^{\beta}_{\infty,\infty}}\, dt'\,.$$
\end{prop}

\begin{proof}
We will follow the proof of \cite{H-K-2008}.
Firstly we can write the solution $\omega$ of the transport equation \eqref{eq:transport} as a sum:
$\omega=\sum_{k\geq -1}\omega_k$, with $\omega_k$ satisfying
\begin{equation}\label{eq:vort_k}\left\{\begin{array}{c}
\d_t \omega_k+v\cdot\nabla\omega_k=\Delta_k g,\\
\omega_k|_{t=0}=\Delta_k \omega_0.
\end{array}
\right.
\end{equation}
We obviously have from above that
\begin{equation}\label{est:vortL}
\|\omega_k(t)\|_{L^\infty}\leq \|\Delta_k \omega_0\|_{L^\infty}\,+\, \Int^t_0 \|\Delta_k g\|_{L^\infty} \, dt'.
\end{equation}
By classical transport estimates in Proposition \ref{p:transp_0}, for any $\epsilon\in (0,1)$, we have
\begin{equation}\label{est:vorteps}
\|\omega_k(t)\|_{B^{ \epsilon}_{\infty,1}}
\leq \Bigl(\|\Delta_k \omega_0\|_{B^{ \epsilon}_{\infty,1}}
\,+\,   \|\Delta_k g\|_{L^1_t(B^{ \epsilon}_{\infty,1})}\Bigr)\exp\Bigl(C\|\nabla v\|_{L^1_t(L^\infty)}\Bigr).
\end{equation}

In order to get a priori estimates in Besov space $B^{-\epsilon}_{\infty,1}$, after applying the operator $\dj$ to Equation
\eqref{eq:vort_k}, we  write the commutator $[v,\dj]\cdot\nabla\omega_k$  as follows
(recalling Bony's decomposition \eqref{eq:bony} and denoting ${\tilde v}:=v-\Delta_{-1}v$)
\begin{eqnarray*}
& & [T_{\tilde v},\dj]\cdot\nabla\omega_k+T_{\dj\nabla\omega_k}{\tilde v}+R(\dj\nabla\omega_k, {\tilde v})
-\dj(T_{\nabla\omega_k} {\tilde v}) \\
& & \qquad\qquad\qquad\qquad\qquad\qquad
-\dj\div(R(\omega_k,{\tilde v}))+\dj R(\omega_k,\div {\tilde v})
+[\Delta_{-1}v,\dj]\cdot\nabla\omega_k.
\end{eqnarray*}
Then, for all $\beta > \epsilon$, the $L^\infty$-norm of all the above terms can be bounded by
(for some nonnegative sequence $\|(c_j)\|_{\ell^1}=1$):
$$
C(d,\beta)\,2^{-j\epsilon}\,c_j\, \cV'(t) \|\omega_k\|_{B^{-\epsilon}_{\infty,1}}.
$$
Thus, we have the following a priori estimate in the space  $B^{-\epsilon}_{\infty,1}$:
\begin{equation}\label{est:vort-eps}
\|\omega_k(t)\|_{B^{ -\epsilon}_{\infty,1}}
\leq \Bigl(\|\Delta_k \omega_0\|_{B^{ -\epsilon}_{\infty,1}}
\,+\,   \|\Delta_k g\|_{L^1_t(B^{- \epsilon}_{\infty,1})}\Bigr)
\exp\Bigl(C \cV(t)\Bigr).
\end{equation}

On the other side, one has  the following, for some positive integer $N$ to be determined hereafter:
$$
\|\omega\|_{B^0_{\infty,1}}\leq \sum_{j,k\geq -1} \|\dj \omega_k\|_{L^\infty}
=\sum_{|j-k|<N} \|\dj \omega_k\|_{L^\infty}\,+\,
\sum_{|j-k|\geq N} \|\dj \omega_k\|_{L^\infty}.
$$
 Estimate \eqref{est:vortL} implies
$$
\sum_{|j-k|<N} \|\dj \omega_k\|_{L^\infty}
\leq N\sum_k \Bigl(\|\Delta_k \omega_0\|_{L^\infty}\,+\,  \|\Delta_k g\|_{L^1_t(L^\infty)} \Bigr)
\leq N \Bigl(\| \omega_0\|_{B^0_{\infty,1}}\,+\,  \| g\|_{L^1_t(B^0_{\infty,1})} \Bigr),
$$
while Estimates \eqref{est:vorteps} and \eqref{est:vort-eps} entail the following (for some nonnegative sequence $(c_j)\in \ell^1$):
$$
  \|\dj \omega_k\|_{L^\infty}
\leq 2^{-\epsilon|k-j|}  c_j
 \Bigl(\|  \Delta_k \omega_0\|_{L^\infty}\,+\,  \| \Delta_k g\|_{L^1_t(L^\infty)} \Bigr)
\exp\Bigl(C\cV(t)\Bigr),
$$
which issues immediately
$$
\sum_{|j-k|\geq N} \|\dj \omega_k\|_{L^\infty}
\leq 2^{-N\epsilon}   \Bigl(\|  \omega_0\|_{B^0_{\infty,1}}\,+\,  \| g\|_{L^1_t(B^0_{\infty,1})} \Bigr)
\exp\Bigl(C\cV(t)\Bigr).
$$

Therefore, for any $\beta>0$, we can choose   $\epsilon\in (0,1)$ and $N\in \N$ such that $\epsilon<\beta$ and $N\epsilon\log 2\sim 1+C\cV(t)$.
Thus the lemma follows from the above estimates.
\end{proof}

%%%%%%%%%%%%%%%%%%%%%%%%%%%%%%%%%%%%%%%%%%%%%%%%%%%%%%%%%%%%%%%%%%%%%%%%%%%%%%%%%%%%%%%%%%%%%%%%%%
%%%%%%%%%%%%%%%%%%%%%%%%%%%%%%%%%%%%%%%%%%%%%%%%%%%%%%%%%%%%%%%%%%%%%%%%%%%%%%%%%%%%%%%%%%%%%%%%%%
\section{Proof of the main results} \label{s:proofs}
%%%%%%%%%%%%%%%%%%%%%%%%%%%%%%%%%%%%%%%%%%%%%%%%%%%%%%%%%%%%%%%%%%%%%%%%%%%%%%%%%%%%%%%%%%%%%%%%%%%%%%%%%%%%%%%%%%%%%%
%%%%%%%%%%%%%%%%%%%%%%%%%%%%%%%%%%%%%%%%%%%%%%%%%%%%%%%%%%%%%%%%%%%%%%%%%%%%%%%%%%%%%%%%%%%%%%%%%%%%%%%%%%%%%%%%%%%%%%

We are now ready to tackle the proof our main results, which this section is devoted to.

First of all we will focus on the proof of Theorem \ref{thm:L2wp}; in the second part, instead,
we will deal with Theorem \ref{th:2D_life}.

%%%%%%%%%%%%%%%%%%%%%%%%%%%%%%%%%%%%%%%%%%%%%%%%%%%%%%%%%%%%%%%%%%%%%%%%%
\subsection{Proof of the local in time well-posedness result  } \label{s:any_p}
%%%%%%%%%%%%%%%%%%%%%%%%%%%%%%%%%%%%%%%%%%%%%%%%%%%%%%%%%%%%%%%%%

In this subsection we will prove Theorem  \ref{thm:L2wp}.
 We will follow the standard procedure: in Step 1 we construct a sequence of approximate solutions having uniform
 bounds, and in Step 2 we prove the convergence of this sequence.

For the sake of conciseness, we will present the proof just for $r=1$, for which we can use classical time-dependent spaces
$L^q_T(B^s_{\infty,1})$. The general case is just more technical, but it doesn't involve
any novelty: it can be treated as in \cite{F-L_p}, by use of refined commutator and product estimates in Chemin-Lerner spaces.

\medbreak
Let us make some simplifications in the coming proof.
We always suppose the existence time $T^\ast\leq 1$ and that
all the constants appearing in the sequel,  such as $C,C_M,C_E$, are bigger than $1$.
We always denote $f^n=f(\rho^n)$ and $\delta f^n=f(\rho^n)- f(\rho^{n-1})$ for $f=f(\rho)$.

\subsubsection*{Step 1: construction of a sequence of approximate solutions}%%%%%%%%%%%%%%%
%%%% L2 RHO %%%%%%%%%%%%%%%%%%%%%%%%%%%%%

 As usual, after  fixing $(\vrho^0, u^0,\nabla\pi^0)=(\vrho_0,  u_0,0)$, we consider inductively
the $n$-th approximate solution $(\vrho^n, u^n, \pi^n)$ to be the unique global solution of the  following linear system:
\begin{equation}\label{L2seqeq}
\left \{\begin{array}{cc}
&\d_{t}\varrho^{n}+u^{n-1}\cdot\nabla\varrho^{n}-\div(\kappa^{n-1}\nabla\varrho^{n})=0,\\
&\d_{t}u^{n}+(u^{n-1}- \kappa^{n-1}(\rho^{n})^{-1}\nabla \rho^n)\cdot\nabla u^{n}+\lambda^{n}\nabla\pi^{n}=h^{n-1},\\
&\div u^{n}=0,\\
&  (\varrho^{n},  u^{n})|_{t=0}= (\varrho_{0}, u_{0}),
\end{array}
\right .
\end{equation}
where $b^{n-1}=b(\rho^{n-1})$, $\kappa^{n-1}=\kappa(\rho^{n-1})$, $\lambda^{n-1}=\lambda(\rho^{n-1})$  and
\begin{equation}\label{L2seqeq:h}
h^{n-1} =(\rho^{n-1})^{-1}\Bigl( \Delta b^{n-1} \nabla a^{n-1}
    \,+\,  u^{n-1}\cdot \nabla^2 a^{n-1}
   \, + \,\nabla b^{n-1}\cdot\nabla^2 a^{n-1} \Bigr)\,,\quad a^{n-1}=a(\rho^{n-1}).
\end{equation}
It is easy to see that by testing $\eqref{L2seqeq}_2$ by $\rho^n u^n$, one should has the energy identity for $u^n$
\begin{equation}\label{L2seq:u,energy}
\frac 12\frac{d}{dt}\Int_{\R^d} \rho^n |u^n|^2=\Int_{\R^d} \rho^n h^{n-1}\cdot u^n.
\end{equation}

%%%%%%%%%%%%%%%%%%%%% A PRIORI ESTIMATE %%%%%%%%

In this paragraph, one denotes
$$
M:=\|(\varrho_0, u_0)\|_{B^s_{\infty,1}},
\quad E_0:=\|\varrho_0\|_{L^2}+\|u_0\|_{L^2}.
$$
We aim at proving that there exist a sufficiently small
parameter $\tau$ (to be determined later),
a positive time $T^{*}$ (which may depend on $\tau$),
a positive constant $C_M$ (which may depend on $M$) and a positive constant $C_E$ such that  the uniform estimates for
the solution sequence $(\rho^n,u^n,\nabla\pi^n)$ hold:
\begin{align}
&
 \rho_\ast \leq \rho^{n-1}:=1+\varrho^{n-1}\,,
  \|\varrho^{n-1}\|_{  L^{\infty}_{T^\ast}(  B^s_{\infty,1} )}\leq C_M\,,
\quad
\|\varrho^{n-1}\|_{  L^{2}_{T^\ast}(  B^{s+1}_{\infty,1} ))\cap  L^{1}_{T^\ast}(  B^{s+2}_{\infty,1} )}
       \leq \tau\,,\label{seqest:rho}\\[1ex]
&\|u^{n-1}\|_{  L^{\infty}_{T^\ast}(  B^s_{\infty,1} ) }\leq C_M,\;
\|u^{n-1}\|_{  L^{2}_{T^\ast}(  B^s_{\infty,1} )\cap   L^{1}_{T^\ast}(  B^s_{\infty,1} ) }\leq \tau,
\;\|\nabla\pi^{n-1}\|_{  L^{1}_{T^\ast}(  B^s_{\infty,1}   \cap  L^2)}\leq \tau^{1/2(d+2)},\label{seqest:u}\\[1ex]
&\|\varrho^{n }\|_{L^\infty_{T^\ast}(L^2)}+\|\nabla\varrho^{n }\|_{L^2_{T^\ast}(L^2)}
+\|u^{n }\|_{L^\infty_{T^\ast}(L^2)}\leq C_E E_0 .\label{L2seqest}
\end{align}

Firstly,   by choosing small $T^\ast$, Estimates
\eqref{seqest:rho}, \eqref{seqest:u}  and \eqref{L2seqest} all hold true for $n=0$.
Next we
 suppose the $(n-1)$-th element $(\varrho^{n-1},u^{n-1},\nabla\pi^{n-1}) $ to belong to Space $E$, defined as the set of the
 triplet $(\varrho, u, \nabla\pi)$ belonging to
\begin{equation}\label{space:E,L2}
   \Bigl(
      C(\R^+;B^s_{\infty,1}\cap L^2)\cap L^2_\loc(H^1)\cap \tilde L^1_\loc(B^{s+2}_{\infty,1})
      \Bigr)
 \times
       \Bigl(
    C(\R^+;B^s_{\infty,1} \cap  L^2)
    \Bigr)
 \times
  \Bigl(   L^1_\loc (B^s_{\infty,1})\cap L^1_\loc( L^2) \Bigr),
\end{equation}
such that  the   inductive assumptions are satisfied.
We then just have to show that the $n$-th unknown
$(\varrho^{n }, u^{n },\nabla\pi^{n })$  defined by System \eqref{L2seqeq} belongs to the same space $E$.

\smallbreak%%%%%%%%%%%%%%%%L2SEQEST RHO %%%%%%%%%%%%%%%%%%%%%%%%%%%

According to Proposition \ref{prop:heat-holder} (or Remark \ref{rem:rho}),  $\varrho^n$  belongs to
$C(\R^+;B^s_{\infty,1})\cap  L^1_t(B^{s+2}_{\infty,1})$ for any finite $t>0$ .
On the other hand, since $u^{n-1}\in L^\infty_\loc(L^\infty)$, energy inequality \eqref{L2:rho} for $\vrho^n$ follows and
$\vrho^n\in C(\R^+;L^2)\cap L^2_\loc(H^1)$.

 As in \cite{F-L_p}, we introduce $\vrho_L$ to be the solution
of the free heat equation with initial datum $\vrho_0$,    which satisfies
 \begin{align}
&\|\varrho_{L}\|_{  L^{\infty}_{T }(B^s_{\infty,1} )}
\,+\,\|\varrho_{L}\|_{  L^{1}_{T }( B^{s+2}_{\infty,1} )}\;\leq\;
C_T\,\|\varrho_0\|_{B^s_{\infty,1}}\,,\quad\forall T>0,\quad C_T \hbox{ depends on }T, \label{seqest:rho,L}\\
& \|\varrho_L\|_{  L^{2}_{T^\ast}( B^{s+1}_{\infty,1} )\cap L^{1}_{T^\ast}( B^{s+2}_{\infty,1})}\leq \tau^2
\quad \hbox{ for small enough }T^\ast.\label{seqest:rho,L,tau}
\end{align}
  Correspondingly, the remainder $\bar\vrho^n:=\vrho^n-\vrho_L$ verifies the following system:
\begin{equation}\label{seqeq:theta,bar}
\left \{
\begin{array}{cc}
&\d_{t}\bar{\varrho}^{n}+u^{n-1}\cdot\nabla \bar{\varrho}^{n}-
\div (\kappa^{n-1} \nabla \bar{\varrho}^{n})=-u^{n-1}\cdot\nabla \varrho _{L}+\div ((\kappa^{n-1}-1)\nabla\varrho _{L}),\\
&\bar{\varrho}^{n}|_{t=0}=0.
\end{array}
\right .
\end{equation}
 Proposition   \ref{prop:heat-holder} (or Remark \ref{rem:rho}) thus implies that
$$
\|\bar\vrho^n\|_{ L^\infty_t( B^s_{\infty,1} )\cap   L^1_t(B^{s+2}_{\infty,1})}
\leq
  \left( C^{n-1}(t)\, e^{C^{n-1}(t) \cK^{n-1}(t)} \right)
  \|f^n\|_{  L^1_t(B^s_{\infty,1})}  ,
$$
where $C^{n-1}(t)$  depends  on $\|\varrho^{n-1}\|_{L^\infty_t(B^s_{\infty,1})}$, and
\begin{align*}
&\cK^{n-1}(t):=t+\|\nabla\kappa^{n-1}\|_{L^2_t( B^s_{\infty,1} )}^2 ,
\\
&f^n:=-u^{n-1}\cdot\nabla\bar\varrho^n
    -u^{n-1}\cdot\nabla\vrho_L
    +\div((\kappa^{n-1}-1) \nabla\varrho_L).
\end{align*}
Inductive assumptions and product estimates in Proposition \ref{p:prod} entail hence
\begin{align*}
&C^{n-1}(t)\, e^{C^{n-1}(t) \cK^{n-1}(t)}\leq C_\cK, \quad \forall t\in [0,T^\ast],
\, C_\cK \hbox{ depending only on }M,
\\&
\|f^n\|_{ L^1_{T^\ast}(B^s_{\infty,1})}
 \leq  \|u^{n-1}\|_{  L^2_{T^\ast}( B^s_{\infty,1} )}
       \| \nabla\bar\varrho^n\|_{  L^2_{T^\ast}(B^s_{\infty,1})}
       + C C_M\tau^2.
\end{align*}
Therefore by the interpolation inequality
 $$
 \| \nabla\bar\varrho^n\|_{  L^2_{T^\ast}(B^s_{\infty,1})}
 \leq C \| \bar\varrho^n\|_{  L^\infty_{T^\ast}(B^s_{\infty,1})} ^{1/2}
 \| \bar\varrho^n\|_{  L^1_{T^\ast}(B^{s+2}_{\infty,1})} ^{1/2},
 $$
 the following smallness statement   pertaining to $\bar\vrho^n$
 \begin{equation}\label{seqest:theta,bar}
\|\bar{ \varrho }^{n}\|_{  L^{2}_{T^{*}}(B^{s+1}_{\infty,1})}
\leq \|\bar{ \varrho }^{n}\|_{ L^{ \infty }_{T^{*}}(B^s_{\infty,1})}
+\Vert \bar{\varrho}^{n}  \Vert_{ L^{1}_{T^{*}}(B^{s+2}_{\infty,1})}\leq \tau^{3/2}\,.
\end{equation}
is verified.
Hence inductive
assumption \eqref{seqest:rho} holds for $\varrho^n$.

\smallbreak %%%%%%%%%%%%%%%%%%%%% PIn%%%%
 %%%%%%%%%%%%%%%%%%%%%%%%%%%%%%%%%%%%%%%%%%%%%%%%%%%%%%%%%%%%%%%%
We will bound $u^n$ and $\nabla\pi^n$   in the following steps:
\begin{enumerate}[(i)]
\item
Energy Identity \eqref{L2seq:u,energy} holds. For bounding $\|h^{n-1}\|_{L^2}$ we use the following inequalities
 \begin{align*}
&\|\Delta b\nabla a\|_{L^2}
\lesssim \|b\|_{B^{s+1}_{\infty,1}} \|\nabla a\|_{L^2}
\lesssim \|\varrho\|_{B^{s+1}_{\infty,1}}\|\nabla\rho\|_{L^2}, \\
&\|u\cdot\nabla^2 a\|_{L^2}
\lesssim \| u\|_{L^2} \|\varrho\|_{B^{s+1}_{\infty,1}}  .
\end{align*}
 Thus the induction assumptions imply  \eqref{L2seqest}.

\item
Standard estimates for transport equation in Besov spaces and inductive assumptions ensure that
\begin{equation}\label{L2seqest:u}
\|u^n\|_{ L^\infty_t(B^s_{\infty,1})}
 \leq
    CC_M e^{C C_M \tau}
   \left( \|u_0\|_{B^s_{p,r}}
    \,+\, \tau +\|\nabla\pi^{n }\|_{  L^1_t(B^s_{\infty,1})}
    \right)
    \leq CC_M( 1  +  \Pi^n ),
 \end{equation}
 where we have defined
 $$
 \Pi^n\,:=\,\|\nabla\pi^n\|_{  L^1_{T^\ast}(B^s_{\infty,1})}.
 $$

 \item %%%%%%%%%%%%%%%Ã¹%%%%%%%%%%%%%L2 seqest Pi-L2 %%%%%%%%%%%%%%%%%%%
Consider the elliptic equation satisfied by $\pi^n$:
 \begin{equation*}\label{L2seqeq:pi}
\div(\lambda^n\nabla\pi^n)\,=\,
\div \Bigl( h^{n-1}
-    (u^{n-1}- \kappa^{n-1}(\rho^{n})^{-1}\nabla \rho^n)\cdot\nabla   u^{n} \Bigr).
\end{equation*}
By view of $\|\nabla u^n\|_{L^2_{T^\ast}(L^\infty)}\leq (T^\ast)^{1/2}\|u^n\|_{L^\infty_{T^\ast}(B^s_{\infty,1})}$,
inductive assumptions  \eqref{seqest:u}   and \eqref{L2seqest} imply
\begin{align}\label{L2seqest:pi,L2}
\|\nabla\pi^n\|_{L^1_{T^\ast}(L^2)}
&\leq
C\Bigl\| h^{n-1} -    (u^{n-1}- \kappa^{n-1}(\rho^{n})^{-1}\nabla \rho^n)\cdot\nabla   u^{n}  \Bigr\|_{L^1_{T^\ast}(L^2)} \nonumber \\
&\leq CC_E E_0 ( \tau +  \tau \|u^n\|_{L^\infty_{T^\ast}(B^s_{\infty,1})}) \nonumber\\
&\leq CC_E C_M E_0 \tau   (1+ \Pi^n  ),\ \hbox{ if }  (T^\ast)^{1/2}\leq \tau.
\end{align}

\item %%%%%%%%%%%%%%%%%%%%% L2seqest Pi s-1 %%%%%%%%%%%%%%%%%%%%%%%%%%%%%%%%%%
Now, interpolation between $B^s_{\infty,1}$ and $L^2$   entails hence (with some appropriated $C_\Pi$
and $0<\g<1$)
\begin{align*}
\|\nabla\pi^{n }\|_{ L^1_{T^\ast}(B^{s-1/2}_{\infty,1})}
 \leq
     C\|\nabla\pi^{n }\|_{  L^1_{T^\ast}(L^2)}^{\g}
      \|\nabla\pi^{n }\|_{ L^1_{T^\ast}(B^s_{\infty,1})}^{1-\g}
 \leq
      C_\Pi(1+\Pi^n) \tau^{\g} .
\end{align*}
\item %%%%%%%L2 seqest pi%%%%Ã¹

Let's consider the following equation
\begin{equation*}\label{L2seqeq:pi,Lap}
\Delta\pi^n
\,=\,
   \nabla\log\rho^n\cdot\nabla\pi^n
  + \rho^n \div \Bigl( h^{n-1}
        -    (u^{n-1}- \kappa^{n-1}(\rho^{n})^{-1}\nabla \rho^n)\cdot\nabla   u^{n} \Bigr).
\end{equation*}
According to product estimates in Proposition \ref{p:prod}, for  some appropriated   $C_\Pi$,
\begin{equation}\label{L2seqest:Delta}
\|\Delta\pi^n\|_{ L^1_{T^\ast}( B^{s-1}_{\infty,1} )}
\leq
        CC_M \|\nabla\pi^n\|_{  L^1_{T^\ast}(B^{s-1/2}_{\infty,1})}
          +CC_M \tau +  CC_M \tau ^2
\leq C_\Pi (1+\Pi^n) \tau^{\g} .
\end{equation}
Notice that such an inequality is true also in the endpoint case $s=1$, for which we have to estimate the $B^0_{\infty,1}$ norm
of $\Delta\pi^n$. This space is no more an algebra, but we can overcome the problem using the $B^{1/2}_{\infty,1}$ regularity.

\item
By decomposing $\nabla\pi^n$ into low frequency part and high frequency part (and using Bernstein's inequality Lemma \ref{lpfond}), one has
$$
\Pi^n\lesssim \|\nabla\Delta_{-1}\pi^n\|_{L^1_{T^\ast}(L^2)}
 + \|\Delta \pi^n\|_{L^1_{T^\ast} (B^{s-1}_{\infty,1})}
 \lesssim \|\nabla \pi^n\|_{L^1_{T^\ast}(L^2)}
 + \|\Delta \pi^n\|_{L^1_{T^\ast} (B^{s-1}_{\infty,1})}.
$$
Thus, the above two estimates  \eqref{L2seqest:pi,L2} and \eqref{L2seqest:Delta}
imply, for $\tau$ and $T^\ast$ small enough, the inductive assumption  \eqref{seqest:u} for $\pi^n$.
Furthermore, \eqref{L2seqest:u} entails the inductive assumption  \eqref{seqest:u} for $u^n$.
\end{enumerate}

\subsubsection*{Step 2: convergence of the approximate solution sequence}\label{s:L2conv}%%%%%%  L2-Cauchy Sequence  %%%%%%%%%%%%%%%%%%%%%%%%%%
%%%%%%%%%%%%%%%%%%%%%%%%%%%%%%%%%%%%%%%%%%%%%%%%%%%%%%%%%%%%%%%%%%%%%%%%%%%%%%%%%%%%%%%%%%%%%%%%%%%%%

Let us turn to establish that the above sequence converges to the solution. Let's introduce the difference sequence
$$
(\delta\vrho^n,\,\delta u^n,\,\nabla\delta\pi^n)
   \,=\,(\vrho^n-\vrho^{n-1},\, u^n-u^{n-1},\, \nabla\pi^n-\nabla\pi^{n-1}),
   \quad n\geq 1.
$$
When $n\geq 2$, it verifies the following system:
\begin{equation}\label{L2seqeq:diff}
\left \{
\begin{array}{cc}
&\d_{t}\delta \varrho^{n}+u^{n-1}\cdot\nabla\delta \varrho^{n}
-\div(\kappa^{n-1}\nabla\delta \varrho^{n})=F ^{n-1},\\
&\d_{t}\delta u^{n}
       +( u^{n-1}-\kappa^{n-1}\,\nabla\log\rho^n )\cdot\nabla\delta u^{n}
          +\lambda^{n }\nabla\delta\pi^{n} \, =\, H_e^{n-1},\\
&\div\delta u^{n}=0,\\
&(\delta \varrho^{n},\delta u^{n})|_{t=0}=(0,0),
\end{array}\right .
\end{equation}
where we have set
\begin{align*}
& F ^{n-1}=
        -\delta u^{n-1}\cdot\nabla\varrho^{n-1}
        +\div(\delta\kappa^{n-1} \nabla\varrho^{n-1}),\\
& H_e^{n-1}=\delta h^{n-1}
             -(\delta u^{n-1}
                 -\delta\kappa^{n-1}\nabla\log\rho^n
                 \, -\, \kappa^{n-2} \nabla \delta(\log\rho)^n )\cdot\nabla u^{n-1}
        - \delta\lambda^{n } \nabla\pi^{n-1},
\end{align*}
with $\delta h^{n-1} = h^{n-1}-h^{n-2}$.

%%%%%%%%%%%%%%%%%%%%%%%% L2 DIFF SEQ CONVERGE  OBSERVATION %%%%%%%%%%%%%%

 We will consider the difference sequence in the energy space.
 Let's do some analysis first: one needs $H^{n-1}_e$ in $L^1_{T^\ast}(L^2)$ and hence
$$
(\rho^{n-1})^{-1}\Delta\delta b^{n-1} \nabla a^{n-1}
\quad\hbox{ and }\quad
(\rho^{n-1})^{-1}\nabla b^{n-2}\cdot\nabla^2\delta a^{n-1} \hbox{ in }L^1_{T^\ast}(L^2).
$$
We only have $\nabla \varrho^n$ in $L^\infty_{T^\ast}(L^\infty)$, and thus we
need $\nabla^2\delta \varrho^n$ in $L^1_{T^\ast}(L^2)$:  this property
follows from the energy inequality of the equation of $\nabla\delta\varrho^n$
\begin{equation}\label{L2diffseqeq:drho}
 \d_{t}\nabla\delta \varrho^{n}
  +u^{n-1}\cdot\nabla^2\delta \varrho^{n}
 -\div(\kappa^{n-1}\nabla^2\delta \varrho^{n})
=-\nabla\delta\vrho^n\cdot\nabla u^{n-1}
   +\div(\nabla\delta\vrho^n\otimes \nabla\kappa^{n-1})
   +\nabla F^{n-1}\,.
\end{equation}
In the above, the first two terms of the right-hand side are of lower order, while the third one is in \textit{$L^2_\loc(H^{-1})$},
thus taking $L^2$ inner product works.

\smallbreak%%%%%%%%%%%%%%%%%%%%%%%%%%%%%%%%%%%%%%%%%%%%%%%%%%%%%%%%%%%%%%%%%%%%%%%%%%%%%%%%%%%%%
%%%%%%%%%%%%%%%%%%%%%%%%%%%%%%%%%%%%%%%%%%%%%%%%%%%%%%%%%%%%%%%%%%%%%%%%%%%%

Now we begin to make the above analysis in detail.\\
Since $\delta\varrho^n\in E$, the energy equality for Equation $\eqref{L2seqeq:diff}_1$  holds
for $n\geq 2$:
\begin{align*}
\frac 12 \frac{d}{dt}\Int_{\R^d} |\delta\vrho^n|^2
& + \Int_{\R^d} \kappa^{n-1} |\nabla\delta\vrho^n|^2
\\
& \,=\,
 -\Int_{\R^d} \delta u^{n-1}\cdot\nabla\rho^{n-1} \,\delta\vrho^n
 -\Int_{\R^d} \delta \kappa^{n-1}\nabla\rho^{n-1}\cdot\nabla\delta\vrho^n,\,.
\end{align*}
Thus integration in time  and the uniform estimates for solution sequence  give
\begin{align}\label{L2diffseq:rho}
\|\delta\vrho^n\|_{L^\infty_{T^\ast}(L^2)}
&+ \|\nabla\delta\vrho^n\|_{L^2_{T^\ast}(L^2)}
\nonumber \\
& \leq
   C  (\|\delta \vrho^{n-1}\|_{L^\infty_{T^\ast}(L^2)}
       + \|\delta u^{n-1}\|_{L^2_{T^\ast}(L^2)})\,\tau.
\end{align}

%%%%%%%%%%%% L2converge d rho%%%%%%%%%%%%%

Similarly, energy equality holds for  $\nabla\delta\vrho^n$, $n\geq 2$ (in fact, it's not clear that $\delta\varrho^1\in L^2_\loc(H^2)$):
 $$\displaylines{
\hspace{1cm}  \frac 12\frac{d}{dt} \Int_{\R^d} |\nabla\delta\vrho^n|^2
 + \Int_{\R^d} \kappa^{n-1} |\nabla^2\delta\vrho^n|^2
\hfill \cr
=
- \Int \nabla\delta\vrho^n\cdot \nabla u^{n-1}\cdot\nabla\delta\vrho^{n }
+ \nabla\delta\vrho^n\cdot \nabla^2\delta\vrho^n\cdot\nabla\kappa^{n-1}
+ F^{n-1} \Delta\delta\vrho^n .
 }$$
  Integrating in time and the inductive assumptions also imply
\begin{align*}
  \|\nabla\delta\vrho^n\|_{L^\infty_{T^\ast}(L^2)}
&+ \|\nabla^2\delta\vrho^n\|_{L^2_{T^\ast}(L^2)}
 \\
& \leq
   C\tau  \|(\delta \vrho^{n-1}, \delta u^{n-1})\|_{L^\infty_{T^\ast}(L^2)}
   + CC_M\|\nabla\delta\varrho^{n-1}\|_{L^2_{T^\ast}(L^2)}.
\end{align*}
By controlling  $\|\nabla\delta\varrho^{n-1}\|_{L^2_{T^\ast}(L^2)}$ above by \eqref{L2diffseq:rho},
one sums up these two inequalities, entailing
\begin{align}\label{L2diffseqest:rho}
\| \delta\vrho^n\|_{L^\infty_{T^\ast}(H^1)}
&+ \|\nabla\delta\vrho^n\|_{L^2_{T^\ast}(H^1)} \nonumber\\
&\leq
   CC_M\tau \|(\delta \vrho^{n-1},  \delta \vrho^{n-2}, \delta u^{n-1}, \delta u^{n-2})\|_{L^\infty_{T^\ast}(L^2)}.
\end{align}

\smallbreak  %%%%%%%%%%%%%%%%%%%%% L2 seq-Cauchy u%%%%%%%%%%%%%%%%%%%%

Now we turn to $\delta u^n$. We rewrite $\delta h^{n-1}$ as
$$\displaylines{
\frac{1}{\rho^{n-1}} \Bigl(
      \Delta\delta b^{n-1} \nabla a^{n-1}
         +\Delta b^{n-2}  \nabla\delta a^{n-1}
    + \delta u^{n-1}\cdot\nabla^2 a^{n-1}
    \hfill\cr\hfill
        +   u^{n-2}\cdot\nabla^2\delta a^{n-1}
        +\nabla\delta b^{n-1}\cdot \nabla^2 a^{n-1}
         +\nabla b^{n-2} \cdot\nabla^2\delta a^{n-1}
\Bigr )
\hfill\cr\hfill
    +\Bigl( (\rho^{n-1})^{-1} - (\rho^{n-2})^{-1} \Bigr)
       (\Delta b^{n-2}\cdot \nabla  a^{n-2}
         +  u^{n-2}\cdot\nabla^2 a^{n-2}
          +\nabla b^{n-2} \cdot\nabla^2  a^{n-2}).
}$$
From the inductive estimates we also have that
\begin{align*}
\|\delta h^{n-1}\|_{L^1_{T^\ast}(L^2)}
&\leq CC_M\tau (\|\delta\vrho^{n-1}\|_{L^2_{T^\ast}(H^2) }
       +\|\delta u^{n-1}\|_{L^\infty_{T^\ast}(L^2)})
       \\
&\qquad
       +CC_E E_0 \tau \|\delta\vrho^{n-1}\|_{L^\infty_{T^\ast}(L^2)},
\end{align*}
and
\begin{align*}
\|H_e^{n-1}\|_{L^1_{T^\ast}(L^2)}
&\leq C(C_M+C_E E_0)\tau (\|\delta\vrho^{n-1}\|_{L^2_{T^\ast}(H^2) }
       +\|\delta u^{n-1}\|_{L^\infty_{T^\ast}(L^2)})
       \\
 &\qquad
       +C\tau^{1/{2(d+2)}} \|\delta\vrho^n\|_{L^\infty_{T^\ast}(L^2)}.
\end{align*}

By view of the density equation for $\rho^n$  and $\div \delta u^n=0$, one has
\begin{equation} \label{L2diffseqest:u}
\|\delta u^n\|_{L^\infty_{T^\ast}(L^2)}
\leq C\|H_e^{n-1}\|_{L^1_{T^\ast}(L^2)}.
\end{equation}
Combining Estimate \eqref{L2diffseqest:rho} and \eqref{L2diffseqest:u} entails, for sufficiently small $\tau$ (depending
only on $d$,  $C_M$, $C_E$, $E_0$),
$$
\| \delta\vrho^n\|_{L^\infty_{T^\ast}(H^1)\cap L^2_{T^\ast}(H^2)}
 + \|\delta u^n\|_{L^\infty_{T^\ast}(L^2)}
 \leq
       \frac 16  \|(\delta\vrho^{n-1}, \delta\vrho^{n-2}, \delta\vrho^{n-3},
       \delta u^{n-1},\delta u^{n-2},\delta u^{n-3}) \|_{L^\infty_{T^\ast}(L^2) }.
$$
Thus $\sum \|(\delta\varrho^n, \delta u^n)\|_{L^\infty_{T^\ast}(L^2)}$ converges. Since $\delta\varrho^n\in C(\R^+;L^2)$,
 the Cauchy sequences  $\{\varrho^n\}$ and $\{u^n\}$ converge respectively to $\varrho$ and $u$ in $C([0,T^\ast];L^2)$.
It is also easy to see that
$$
\sum_{n\geq 2} \| \delta\vrho^n\|_{L^\infty_{T^\ast}(H^1)\cap L^2_{T^\ast}(H^2)},
\quad \sum_{n\geq 2} \|\delta h^n\|_{L^1_{T^\ast}(L^2) } ,
\quad  \sum_{n\geq 2} \| H^{n-1}_e\|_{L^1_{T^\ast}(L^2) }  <+\infty.
$$
Rewrite the elliptic equation for  $\pi^n$
\begin{align*}\label{L2diffseqeq:pi}
\div(\lambda^n \nabla\delta\pi^n)
&=\div H^{n-1}_e \,-\, \div(( u^{n-1}-\kappa^{n-1}\,\nabla\log\rho^n )\cdot\nabla\delta u^{n}),\\
&=\div H^{n-1}_e \,-\,\div \Bigl(   \delta u^{n}\cdot \nabla ( u^{n-1}-\kappa^{n-1}\,\nabla\log\rho^n )
+\delta u^n  \div(\kappa^{n-1}\nabla\log \rho^n)  \Bigr).
 \end{align*}
We hence get
$$
\|\nabla\delta\pi^n\|_{L^1_{T^\ast}(L^2)}
\leq C
      (\|H^{n-1}_e\|_{L^1_{T^\ast}(L^2)} + C_M \|\delta u^n\|_{L^\infty_{T^\ast}(L^2)} ).
$$
Thus $\sum^\infty_2 \|\nabla\delta\pi^n\|_{L^1_{T^\ast}(L^2) } $ also converges and hence $\nabla\pi^n$ converges
to the unique limit $\nabla\pi$ in $L^1_{T^\ast}(L^2)$.

Finally, one easily checks that the limit $(\rho, u,\nabla\pi)$ solves System \eqref{system} and   is in $E(T^\ast)$ by Fatou property.
{}
The proof of the uniqueness is quite similar and we omit it.

%%%%%%%%%%%%%%%%%%%%%%%%%%%%%%%%%%%%%%%%%%%%%%%%%%%%%%%%%%%%%%%%%%%%%%%%%%%%%%%%%%%%%%%%%%%%%%%%%%%%%%%%%%%%%%%%%%%%%%
\subsection{Lower bounds for the lifespan in dimension $d=2$} \label{ss:life}
%%%%%%%%%%%%%%%%%%%%%%%%%%%%%%%%%%%%%%%%%%%%%%%%%%%%%%%%%%%%%%%%%%%%%%%%%%%%%%%%%%%%%%%%%%%%%%%%%%%%%%%%%%%%%%%%%%%%%%

In this  section, we aim to get a lower bound for the lifespan of the solution in the case of dimension $d=2$.
The idea is   to resort to the vorticity in order to control the \textit{high frequencies} of the velocity field,
as done in \cite{D-F} in the context of incompressible Euler equations with variable density.

We define the (scalar) vorticity $\omega$ of the fluid as in the classical case:
\begin{equation} \label{d:vort}
 \omega\,:=\,\d_1u^2\,-\,\d_2u^1\, \equiv\,\d_1v^2\,-\,\d_2v^1\,.
 \end{equation}
According to  $\eqref{eq:EH}_2$, it satisfies the following transport equation:
\begin{equation} \label{eq:vort}
 \d_t\omega\,+\,v\cdot\nabla\omega\,+\,\omega\,\Delta b\,+\,\nabla\lambda\wedge\nabla\Pi\,=\,0\,,
\end{equation}
where   (recalling the change of variables \eqref{relation:a,b} and \eqref{relation:u})
 $$\displaylines{
 v=u+\nabla b,\quad
 a=a(\rho),\, b=b(\rho),\, \lambda=\lambda(\rho),
 \cr
 \nabla\Pi=\nabla\pi+\nabla\d_t a,\quad
\nabla\lambda\wedge\nabla\Pi\,=\,\d_1\lambda\,\d_2\Pi\,-\,\d_2\lambda\,\d_1\Pi.
}$$

The key to the proof of Theorem \ref{th:2D_life} will be bounding, by use of Proposition \ref{p:transp_0}, the
$B^0_{\infty,1}$ norm of the vorticity \textit{linearly} (but not exponentially)
with respect to the velocity field.

Similarly as in \cite{F-L_p}, let us introduce the following notations:
\begin{align*}
&R_0=\|\varrho_0\|_{B^1_{\infty,1}},
\quad U_0=\|u_0\|_{B^1_{\infty,1}},\\
&R(t)=\|\varrho\|_{L^\infty_t(B^1_{\infty,1})},\quad
S(t)=\|\varrho\|_{L^1_t(B^3_{\infty,1})},\quad
U(t)=\|u\|_{L^\infty_t(B^1_{\infty,1})}.
\end{align*}

First of all, we apply Proposition \ref{prop:heat-holder} to the density equation $\eqref{system}_1$: it's easy to see that we get (noticing $S'(t)=\|\varrho(t)\|_{ B^3_{\infty,1} }$)
\begin{align*}
R+S
&\leq C(1+R^3)\Bigl( R_0 + \int^t_0 UR +UR^{1/2} (S')^{1/2}
+(1+R^{(2-\epsilon)/(1+\epsilon)} (S')^{\epsilon/(1+\epsilon)}) R+R^{3/2} (S')^{1/2}
\Bigr).
\end{align*}
Hence by use of Young's inequality, one arrives at
\begin{align*}
R+S
&\leq
C(1+R^3)R_0
+C(1+R^6)\int^t_0 \left(UR + U^2 R+R+R^3\right)d\tau.
\end{align*}
If we define now
\begin{equation} \label{def_life:T_R}
 T_R\;:=\;\sup\Bigl\{t\,>\,0\;\;\bigl|\;
  R^6\leq 1,\;\int^t_0 R^3(\tau)\,d\tau\;\leq\;2\,R_0\Bigr \}\,,
\end{equation}
 then, for all $t\in [0,T_R]$ we find
\begin{equation}\label{est_life_2D:R+S}
R+S\leq CR_0\exp\Bigl(C\int^t_0 (1+U^2)\Bigr)\,.
\end{equation}

We now estimate the velocity field.
Let us  summarise the following inequalities for the non-linear terms in the momentum equation,
which will be frequently used in the sequel:
\begin{align}
&\|\nabla^2 b(\rho)\|_{B^1_{\infty,1}}
\lesssim \|b\|_{B^3_{\infty,1}}
\lesssim \|\varrho\|_{B^3_{\infty,1}}=S'; \label{est_life_2D:b}\\
&\|\Delta b\nabla a\|_{L^2}
\lesssim \|b\|_{B^2_{\infty,1}} \|\nabla a\|_{L^2}
\lesssim \|\varrho\|_{B^2_{\infty,1}}\|\nabla\rho\|_{L^2}
\lesssim R^{1/2}(S')^{1/2} \|\nabla\rho\|_{L^2}
\leq R\|\nabla\rho\|_{L^2}^2+S' \label{est_life_2D:b,a,L2};\\
&\|(u+\nabla b)\cdot\nabla u\|_{L^2}
\lesssim \|\nabla u\|_{L^\infty}(\|\nabla\rho\|_{L^2}+\|u\|_{L^2})
\lesssim U(\|\nabla\rho\|_{L^2}+\|u\|_{L^2}).\label{est_life_2D:u,L2}
\end{align}
Similarly as the above inequality \eqref{est_life_2D:b,a,L2}, one has also
\begin{equation}\label{est_life_2D:L2}
\|\nabla b\cdot\nabla^2 a\|_{L^2}\lesssim R\|\nabla\rho\|_{L^2}^2+S',
\quad \|u\cdot\nabla^2 a\|_{L^2}\lesssim R\|u\|_{L^2}^2+S'.
\end{equation}

 %%%%%%%%%%%%%%%%%%%%% U %%%%%%%%%%%%%%%%%%%%%%%%%%
Now, by separating low and high frequencies, we find the following bound for the velocity:
\begin{equation} \label{est_life_2D:low-high}
 U(t)\,\leq\,C\left(\|u\|_{L^2}\,+\,\|\omega\|_{B^0_{\infty,1}}\right).
\end{equation}
From the energy inequality for equation \eqref{eq:u} of $u$, i.e.
$$
\|u(t)\|_{L^2}\,\leq\,C\left(\|u_0\|_{L^2}\,+\,\int^t_0\|\div\left(v\otimes\nabla a\right)\|_{L^2}\,d\tau\right)\,,
$$
due to Inequalities \eqref{est_life_2D:b,a,L2} and \eqref{est_life_2D:L2}, it follows that
\begin{equation} \label{est_life_2D:u_2}
 \|u(t)\|_{L^2}\,\leq\,C\left(\|u_0\|_{L^2}
 \,+\,\int^t_0\Bigl(R(\|\nabla\rho\|^2_{L^2}+ \|u\|^2_{L^2})\,+\,S'\Bigr)\,d\tau\right).
\end{equation}

%%%%%%%%%%%%%%
Now, applying Proposition \ref{p:transp_0} with $\beta=1$ to Equation \eqref{eq:vort}, we find
\begin{eqnarray*}
 \|\omega(t)\|_{B^0_{\infty,1}}
  \lesssim  \left(\|\omega_0\|_{B^0_{\infty,1}}
  \,+\,\int^t_0 \Bigl\| \nabla\lambda\wedge\nabla\Pi+ \omega\,\Delta b\Bigr\|_{B^0_{\infty,1}} d\tau
 \right)
 \left(1+\int^t_0\left(\|\nabla u\|_{L^\infty}+\|\nabla^2 b\|_{B^1_{\infty,1}} \right)d\tau\right).
\end{eqnarray*}
By use of Bony's paraproduct decomposition (see also \cite{D-F}), one has
\begin{eqnarray*}
\|\nabla\lambda\wedge\nabla\pi\|_{B^0_{\infty,1}} & \lesssim & \|\nabla\rho\|_{B^0_{\infty,1}}\,\|\nabla\Pi\|_{B^0_{\infty,1}} \\
\left\|\omega\,\Delta b\right\|_{B^0_{\infty,1}} & \lesssim &
\|\omega\|_{B^0_{\infty,1}}\,\|\Delta b\|_{B^1_{\infty,1}}\,.
\end{eqnarray*}
Hence, by virtue of the relation $\|\omega\|_{B^0_{\infty,1}}\lesssim U$, we get
\begin{equation} \label{est_life_2D:vort}
 \|\omega(t)\|_{B^0_{\infty,1}}
 \leq   C \left(U_0
 +\int^t_0\left(R \|\nabla\Pi\|_{B^0_{\infty,1}} + US'\right)d\tau\right)
\left(1+\int^t_0\|\nabla u\|_{L^\infty}d\tau + S\right) .
\end{equation}

%%%%%%%%%%%%%%%% Pi%%%%%%%%%%%%%%%%

It remains us to deal with the pressure term. First of all, from the density equation we have
\begin{equation*}\label{est_life_2D:Pi}
\|\nabla\Pi\|_{B^0_{\infty,1}}\,\leq\,\|\nabla\pi\|_{B^0_{\infty,1}}\,+\,\|\d_t\nabla a\|_{B^0_{\infty,1}}\,\lesssim\,
\|\nabla\pi\|_{B^0_{\infty,1}}\,+\,
U\,S'\,+\,S' \,.
\end{equation*}
We next  bound $ \pi $  which satisfies the following elliptic equation:
\begin{equation*}\label{lineareq:pi,origin}
\div(\lambda\nabla\pi)=\div(h-v\cdot\nabla u)=\div(h-u\cdot\nabla v+u\div v) .
\end{equation*}
Similarly as in Step 1, Subsection \ref{s:any_p}, by decomposing $\nabla\pi$ into the high and low frequency part separately, one arrives at
\begin{align*}
\|\nabla\pi\|_{B^1_{\infty,1}}
&\lesssim
 \|\nabla\Delta_{-1}\pi\|_{L^\infty}+\|\Delta\pi\|_{B^0_{\infty,1}}\\
&\lesssim
\|\nabla \pi\|_{L^2}+\|\lambda^{-1} [-\nabla\lambda\cdot\nabla\pi + \div(h-v\cdot\nabla u)]\|_{B^0_{\infty,1}}\\
&\lesssim
\| h-v\cdot\nabla u \|_{L^2}+ \|\nabla\rho\|_{B^0_{\infty,1}} \|\nabla\pi\|_{B^{\frac 12}_{\infty,1}}\\
&\quad
+(1+\|\nabla\rho\|_{B^0_{\infty,1}}) (\|h\|_{B^1_{\infty,1}}
+\|\div( v\cdot\nabla u )\|_{B^0_{\infty,1}}).
\end{align*}
By the following  interpolation inequality (recall also Lemma \ref{c:embed}),
\begin{align*}
\|\nabla\pi\|_{ B^{1/2}_{\infty,1} }
& \lesssim \|\nabla\pi\|_{  L^2 }^{1/(d+2)}\|\nabla\pi\|_{  B^1_{\infty,1} }^{(d+1)/(d+2)}\,,
\end{align*}
one derives that, for some $\delta>1$,
$$
\|\nabla\pi\|_{B^1_{\infty,1}}\,\leq\,C\left((1+R^\delta)\| h-(u+\nabla b)\cdot\nabla u \|_{L^2}
\,+\,(1+R)
\,\bigl( \|  h \|_{B^1_{\infty,1}}+\| \div(v\cdot\nabla u)\|_{B^0_{\infty,1}} \bigr)\right).
$$
Then, by the product estimates in Proposition \ref{p:prod},
one finally  bounds   $\nabla\pi$ as follows:
$$
\|\nabla\pi\|_{B^1_{\infty,1}}
 \leq C (1+R^\delta )
 \bigl(
R(\|\nabla\rho\|^2_{L^2}+\|u\|^2_{L^2})
 +U(\|\nabla\rho\|_{L^2}+\|u\|_{L^2})
 + (1+R^2)( U S'  +   S' + U^2)\bigr).
$$

Let us define
$$X(t)\,:=\,U(t)+\|u(t)\|_{L^2}\,=\,\|u(t)\|_{L^2\cap B^1_{\infty,1}}.$$
 So we get
\begin{equation*} \label{est_life_2D:pi}
\|\nabla\Pi\|_{B^0_{\infty,1}},\,\|\nabla\pi\|_{B^1_{\infty,1}}
\,\leq\,C\,\left(1+R^{\delta+2}\right)\Bigl(\|\nabla\rho\|^2_{L^2}\,+\,S'\,+\,X^2\,+\, XS'\Bigr).
\end{equation*}
Therefore,  Estimate \eqref{est_life_2D:vort} for the vorticity becomes (denoting $X(0)=X_0$)
\begin{eqnarray*}
\|\omega(t)\|_{B^0_{\infty,1}} & \lesssim & \left(1+S+ \int^t_0X d\tau \right) \\
& & \times \Biggl( X_0+
\int^t_0(1+R^{\delta+3})\biggl(R\,\|\nabla\rho\|^2_{L^2}
\,+\,R\,S'\,+\,R\,X^2\,+\,X\,S'\biggr)d\tau\Biggr)\,.
\end{eqnarray*}
 Keeping in mind \eqref{est_life_2D:R+S} and
  introducing $\ell:=\delta+4>5$,
 from relation \eqref{est_life_2D:low-high} we finally find, for $t\in [0,T_R]$,
\begin{align*}
  X(t)   \leq  C
\Biggl(X_0+
\underbrace{R_0(1+R_0^{\ell}) e^{C\int^t_0(1+X^2)}  } _{\Gamma_1}
\Bigl(  \int^t_0 \|\nabla\rho\|^2_{L^2} + 1 \Bigr)
&\,+\, \underbrace{ (1+R_0^{\ell})
    e^{C\int^t_0(1+X^2)} \int^t_0  XS'\,d\tau }_{\Gamma_2}\Biggr)
    \\
 &   \times\left(1+S
 + \int^t_0Xd\tau \right).
\end{align*}
We define $T_X$ as the quantity
$$
T_X:=\sup\Bigl\{ t\, | \,
 \Gamma_1(t)\leq 1,
\quad \Gamma_2(t)\leq 1+\|\varrho\|_{L^2}^2+X_0 \Bigr\}.
$$
Then, noticing $S\leq \Gamma_1$, one easily arrives at the following bound for $X(t)$, with $t\in[0,T_R]\cap[0,T_X]$:
$$
X(t)\,\leq\,C  (1+\|\vrho_0\|^2_{L^2} +X_0)\left(1+\int^t_0X(\tau)d\tau\right).
$$

Hence, since $\Gamma_0:= (1+\|\vrho_0\|^2_{L^2} +X_0)$, then by Gronwall's lemma we get
$
 X(t)\,\leq\,C\Gamma_0\,e^{C\Gamma_0 t}$.
After a long  but straightforward  calculation (omitted here), one can check that $T$, defined by relation \eqref{est_life_2D:lifespan}
where we take a small enough constant $L$, satisfies
$$
R^6\leq 1,
\quad \int^T_0 R^3\leq 2R_0,
\quad \Gamma_1(T)\leq 1,
\quad  \Gamma_2(T)\leq 1+\|\varrho_0\|_{L^2}^2+X_0.
$$
This completes the proof of Theorem \ref{th:2D_life}.

\bigbreak\bigbreak

\vspace{3cm}
%%%%%%%%%%%%%%%%%%%%%%%%%%%%%%%%%%%%%%%
%ADDRESSES
%%%%%%%%%%%%%%%%%%%%%%%%%%%%%%%%%%%%%%%

\begin{flushleft}
\textsl{Francesco Fanelli} \\ \vspace{.1cm}  
{\small 
\textit{Institut de Math\'ematiques de Jussieu-Paris Rive Gauche -- UMR 7586} \\
 \textsc{Universit\'e Paris-Diderot -- Paris 7} \\
  B\^atiment Sophie-Germain, case 7012 \\
   56-58, Avenue de France \\
   75205 Paris Cedex 13 -- FRANCE \\ \vspace{.1cm}
   E-mail: \texttt{fanelli@math.jussieu.fr} }

 \bigbreak
 \bigbreak

\textsl{Xian Liao} \\ \vspace{.1cm} %\footnote{Present address:
%Academy of Mathematics \& Systems  Science, Chinese Academy of Sciences; 55 Zhongguancun East Road, 100190, Beijing, P.R. China}}\\
{\small \textit{Academy of Mathematics \& Systems Science} \\
 \textsc{Chinese Academy of Sciences}\\
  55 Zhongguancun East Road  \\
  100190 Beijing -- P.R. CHINA \\ \vspace{.1cm}
E-mail: \texttt{xian.liao@amss.ac.cn} }
\end{flushleft}

\end{document}